\documentclass[hidelinks,onefignum,onetabnum]{siamart250211}

\ifpdf
\hypersetup{
  pdftitle={A tensor-based dynamic mode decomposition}, 
  pdfauthor={A.K. Saibaba, M.E. Kilmer, K. Hall-Hooper, F. Tian, A. Mize}
}
\fi
\usepackage{graphicx} 
\usepackage{amsmath,amssymb}
\usepackage{algorithm,algorithmic}
\usepackage[mathlines]{lineno}
\usepackage{subfigure}
\usepackage{tikz}

\newcommand{\R}{\mathbb{R}}
\newcommand{\C}{\mathbb{C}}
\newcommand{\mc}[1]{\mathcal{#1}}
\newcommand{\mb}[1]{\mathbb{#1}}
\newcommand{\B}[1]{\boldsymbol{#1}}

\newcommand{\Bt}[1]{\widetilde{\boldsymbol{#1}}}

\newcommand{\T}[1]{\boldsymbol{\mathcal{#1}}}
\newcommand{\TA}[1]{\overrightarrow{\boldsymbol{\mathcal{#1}}}}
\newcommand{\Th}[1]{\widehat{\boldsymbol{\mathcal{#1}}}}
\newcommand{\Tt}[1]{\widetilde{\boldsymbol{\mathcal{#1}}}}
\newcommand{\bdiag}[1]{\mathsf{bdiag}\left(#1\right)}
\newcommand{\bmat}[1]{\begin{bmatrix}#1\end{bmatrix}}
\newcommand{\starM}{\star_{\B{M}}}
\newcommand{\range}{\mathsf{range}}

\newcommand{\unfold}[1]{\mathsf{unfold}(#1)}
\renewcommand{\t}{^*}
\newcommand{\fold}[1]{\mathsf{fold}(#1)}

\usepackage{tikz} 		
\usepackage{pgfkeys} 	
\usepackage{ifthen} 		
\usetikzlibrary{matrix}

\pgfkeys{
/tensor/.cd, 
dim1/.initial = 1,		dim1/.get = \dimOne,		dim1/.store in = \dimOne,
dim2/.initial = 1,		dim2/.get = \dimTwo,		dim2/.store in = \dimTwo,
dim3/.initial = 1,		dim3/.get = \dimThree,	dim3/.store in = \dimThree,
xshift/.initial = 0, 	xshift/.get = \xShift, 		xshift/.store in = \xShift,
yshift/.initial = 0, 	yshift/.get = \yShift, 		yshift/.store in = \yShift,
xspec/.initial = 0, 	xspec/.get = \xSpec, 		xspec/.store in = \xSpec,
yspec/.initial = 0, 	yspec/.get = \ySpec, 		yspec/.store in = \ySpec,
scale/.initial = 1, 	scale/.get = \myScale, 	scale/.store in = \myScale,
fill/.initial = white, 	fill/.get = \myFill, 		fill/.store in = \myFill,
back edges/.initial = 0, 	back edges/.get = \myBack, 	back edges/.store in = \myBack,
slice type/.initial = none, 		slice type/.get = \sliceType, 			slice type/.store in = \sliceType, 
number of slices/.initial = 1, 	number of slices/.get = \nSlices, 		number of slices/.store in = \nSlices,
slice width/.initial = 1, 		slice width/.get = \sWidth, 				slice width/.store in = \sWidth, 
}

\newcommand{\tensor}[1][]{\@tensor[#1]}
\def\@tensor[#1] (#2,#3) #4; {{ 

\pgfkeys{/tensor/.cd,#1}

\def\depthScale{0.5} 

\pgfmathsetmacro{\numSlicesMinusOne}{\nSlices-1}
\pgfmathsetmacro{\numSlicesPlusOne}{\nSlices+1}

\pgfmathsetmacro{\sliceLength}{\myScale*\dimOne}

\ifthenelse{\equal{\sliceType}{lateral}}
	{
	
	\pgfmathsetmacro{\sliceWidth}{\myScale*\sWidth*0.9*\dimTwo/\nSlices}
	\pgfmathsetmacro{\sliceGap}{\myScale*\dimTwo/(\nSlices-1) - \nSlices*\sliceWidth/(\nSlices-1)}
	\pgfmathsetmacro{\sliceDepth}{\myScale*\dimThree}
	
	} 
	{
	\ifthenelse{\equal{\sliceType}{frontal}}
		{
		
		\pgfmathsetmacro{\sliceDepth}{\myScale*\sWidth*0.9*\dimThree/\nSlices}
		\pgfmathsetmacro{\sliceGap}{\myScale*\dimThree/(\nSlices-1) - \nSlices*\sliceDepth/(\nSlices-1)}
		\pgfmathsetmacro{\sliceWidth}{\myScale*\dimTwo}
	
		}
		{
		\pgfmathsetmacro{\sliceWidth}{\myScale*\dimTwo}
		\pgfmathsetmacro{\sliceDepth}{\myScale*\dimThree}
		}

	}

\def\xFront{#2 + \xShift}	
\def\yFront{#3 + \yShift}
\def\xBack{#2 + \xShift + \depthScale*\sliceDepth + \xSpec*\sliceDepth}
\def\yBack{#3 + \yShift + \depthScale*\sliceDepth + \ySpec*\sliceDepth}

\def\aFront{(\xFront, \yFront)}
\def\bFront{(\xFront, \yFront + \sliceLength)}
\def\cFront{(\xFront + \sliceWidth, \yFront + \sliceLength)}
\def\dFront{(\xFront + \sliceWidth, \yFront)}

\def\aBack{(\xBack, \yBack)}
\def\bBack{(\xBack, \yBack + \sliceLength)}
\def\cBack{(\xBack + \sliceWidth, \yBack + \sliceLength)}
\def\dBack{(\xBack+ \sliceWidth, \yBack)}

\ifthenelse{\NOT\equal{\myFill}{nofill}}
	{
	\def\tempTensor{
		\fill[\myFill!25] \bFront -- \bBack -- \cBack -- \cFront -- cycle; 
		\fill[\myFill!75] \dFront -- \dBack -- \cBack -- \cFront -- cycle; 
		\fill[\myFill!50] \aFront rectangle \cFront;  
	
		\draw \aFront rectangle \cFront; 
		\draw \bFront -- \bBack; 
		\draw \cFront -- \cBack;
		\draw \dFront -- \dBack;
	
		\draw \bBack -- \cBack;
		\draw \cBack -- \dBack;
		}
	}
	{ 

	\def\tempTensor{
		\draw \aFront rectangle \cFront; 
		
		\ifthenelse{\NOT\equal{\myBack}{0}}
		{
			\draw[dashed] \bBack -- \aBack -- \dBack;
		}{}
		
		\draw \dBack -- \cBack -- \bBack;

		\ifthenelse{\NOT\equal{\myBack}{0}}
		{
			\draw[dashed] \aFront -- \aBack;
		}{}
		
		\draw \bFront -- \bBack;
		\draw \cFront -- \cBack;
		\draw \dFront -- \dBack;
		}
	}

\ifthenelse{\equal{\sliceType}{lateral}}
	{
	\foreach\sliceCount in {0,...,\numSlicesMinusOne}
		{	
		\begin{scope}[shift ={(\sliceCount*\sliceWidth + \sliceCount*\sliceGap, 0)}]
			\tempTensor;
		\end{scope}
		}
	
	}
	{
	
	\ifthenelse{\equal{\sliceType}{frontal}}
	{
	
	\pgfmathsetmacro{\xStep}{\sliceDepth/2 + \sliceGap/2 + \myScale*\dimThree*\xSpec/(\nSlices-(1-\sWidth))}
	\pgfmathsetmacro{\yStep}{\sliceDepth/2 + \sliceGap/2 +  \myScale*\dimThree*\ySpec/(\nSlices-(1-\sWidth))}
	
	\foreach\sliceCount in {-\numSlicesMinusOne,...,0}
		{	
		
		\begin{scope}[shift = {(-\sliceCount*\xStep, -\sliceCount*\yStep)}]
			\tempTensor;
		\end{scope}
	
		}
	
	}
	{
	\tempTensor;
	}
	
	}

\node at (#2 + \dimTwo/2, #3 + \dimOne/2) {#4};

}} 

\title{A tensor-based dynamic mode decomposition based on the $\starM$-product\thanks{Submitted to the editors, DATE. \funding{The work was supported, in part, by the Department of Energy through the awards DE-SC0025262 and DE-SC0025415. }}}
\author{Arvind K.\ Saibaba\thanks{Department of Mathematics, North Carolina State University
  (\email{asaibab@ncsu.edu}, 
  \url{https://asaibab.math.ncsu.edu/}).}\and Misha E.\ Kilmer\thanks{Department of Mathematics, Tufts University (\email{misha.kilmer@tufts.edu}, \url{https://sites.tufts.edu/mishaekilmer})}\and Khalil Hall-Hooper\thanks{Department of Mathematics, North Carolina State University
  (\email{khalil.a.hallhooper@protonmail.com})}\and Fan Tian\thanks{Department of Mathematics, Tufts University (\email{fan.tian@tufts.edu})} \and Alex Mize\thanks{Department of Mathematics, North Carolina State University
  (\email{abmize@ncsu.edu})}}

\begin{document}

\maketitle

\begin{abstract}
    Dynamic mode decomposition (DMD) is a data-driven method for estimating the dynamics of a discrete dynamical system. This paper proposes a tensor-based approach to DMD for applications in which the states can be viewed as tensors. Specifically, we use the $\starM$-product framework for tensor decompositions which we demonstrate offers excellent compression compared to matrix-based methods and can be implemented in a computationally efficient manner. We show how the proposed approach is connected to the traditional DMD and physics-informed DMD frameworks. We give a computational framework for computing the tensor-based DMD and detail the computational costs. We also give a randomized algorithm that enables efficient $\starM$-DMD computations in the streaming setting. The numerical results show that the proposed method achieves equal or better accuracy for the same storage compared to the standard DMD on these examples and is more efficient to compute.   
\end{abstract}


\section{Introduction}

Dynamical systems are found in many applications in science and engineering. We consider discrete dynamical systems of the form 
\begin{equation}
\B{x}_{t+1} = \B{F}(\B{x}_t)\qquad t=0,1,2,\dots,
\end{equation}
where $\B{x}_t \in \R^{N}$ denote the states of the system, and $\B{F}$ governs the dynamics of the system. The approach in Dynamical Mode Decomposition (DMD)~\cite{colbrook2023multiverse,kutz2016dynamic} is to suppose that there is a matrix $\B{A}$ such that $\B{x}_{t+1} \approx \B{A}\B{x}_t$ for $ t \ge 0$. To this end, in the regression viewpoint of DMD~\cite[Section 2.2.1]{colbrook2023multiverse}, we solve a least-squares problem of the form 
\begin{equation}\label{eqn:dmdlstsq}
\min_{\B{A} \in \R^{N\times N}} \| \B{Y} - \B{AX}\|_F^2,
\end{equation}
where the matrices $\B{X} = \bmat{\B{x}_0 & \dots & \B{x}_{t-1}} $ and $\B{Y} = \bmat{\B{x}_1 & \dots & \B{x}_{t}}$ are formed from the snapshots.  The solution of this least-squares problem is given by 
\[ \B{A}_{\rm DMD} = \B{YX}^\dagger,\]
where $\dagger$ denotes the Moore-Penrose inverse~\cite[Section 5.5.2]{golub2012matrix}. 
The idea is that once $\B{A}$ is determined, its orthogonal decomposition would enable one to better understand the state dynamics.  
In practice, however,  a truncated SVD of the matrix $\B{X}$ is computed and a Galerkin approximation of the matrix $\B{A}_{\rm DMD}$ is then formed using the subspace spanned by the leading $r$ left singular vectors of the matrix $\B{A}$. This is reviewed in Section \ref{ssec:dmd}.

In applications, the states $\B{x}_t$ are discrete representations of states in two or three spatial dimensions. For example, in fluid dynamics, the states represent a physical quantity such as velocity or vorticity over a two- or three-dimensional space. A disadvantage of DMD is that by vectorizing the states, the approach disregards the spatial structure present in the data.  This suggests that additional structure is present in the problem that is going undiscovered and underutilized. 

A tensor is a multiway array. To avoid vectorizing the states, we can represent them in a tensor, where the modes/dimensions represent the two or three space dimensions and time.  We posit that a tensor-based DMD approach can handle the inherent multidimensional structure of the problem.  However, in order to build such an approach, we need to consider which compressed tensor representation to employ.  

In applications involving data compression, there are many possible tensor decompositions one could use. Examples include the CP, Tucker, and Tensor Train (TT) decomposition (see \cite{ballard2025tensor} for descriptions and references).  However, the aforementioned decompositions do not give a straightforward way to translate the DMD machinery---which is inherently operator based---to tensors (see Related Work at the end of this section for exceptions).  
Instead, we build our tensor-based DMD approach in the tensor $\starM$-product framework \cite{kernfeld2015tensor, kilmer2021tensor}.  
There are several advantages of working in the $\starM$-product framework:
\begin{itemize}
\item  the $\starM$-SVD has provably superior compression properties compared to the matrix representation~\cite{kilmer2021tensor},
\item the $\starM$-product provides a natural way to define a pseudoinverse, enabling a closed-form least-squares solution to a regression problem, analogous to~\eqref{eqn:dmdlstsq},
\item computations in the $\starM$ have inherent parallelism that can be exploited to obtain speedups.
\end{itemize}

 We propose to view the state vectors $\B{x}_t$ as lateral slices $\T{X}_t$ and approximate the dynamical system as $\T{X}_{t+1} \approx \T{A} \starM \T{X}_{t}$, where $\T{A}$ is a tensor to be determined, and the $\starM$ product and related framework is described in detail in Section \ref{ssec:starM}. 
 Next, we consider the streaming setting, in which the snapshots (states) of the dynamical system arrive sequentially in batches and are to be discarded after processing. 
 We show how to update the $\starM$-DMD modes in this streaming setting in a way that avoids recomputing expensive decompositions from scratch.

\medskip
\paragraph{Contributions} 
Specific contributions of this work are summarized as follows:  
\begin{enumerate}
    \item We give a framework for DMD using the $\starM$ product (Section~\ref{sec:tDMD}). In this framework, the states are not viewed as vectors, but lateral slices of a tensor.  We describe how to solve a tensor regression problem to obtain the tensor that maps the current state to the next. 
    \item We explain the connection between $\starM$-DMD to traditional matrix-based DMD (Section~\ref{sec:connections}), thus grounding our method in the Koopman formalism. We show that $\starM$-DMD can be viewed as a matrix regression problem where the operator $\B{A}$ is constrained to lie within a subspace defined by the $\starM$-algebra. 
    \item We give a computational framework for computing $\starM$-DMD (Section~\ref{sec:tDMD}). 
    Specifically, we define a Schur decomposition under the $\starM$ product. Compared with traditional, matrix-based DMD, we show that our proposed approach allows for parallelism and is more computationally efficient.
    \item  To achieve maximum computational efficiency, we propose a new randomized algorithm for $\starM-$SVD and $\starM$-DMD in the streaming setting (Section~\ref{sec:streaming}). The proposed method is storage efficient, only requiring storing two sketches, as well as computationally efficient. We also provide a probabilistic analysis for the $\starM$-SVD algorithm, that may be of independent interest beyond this paper. 
    \item We demonstrate the performance of our approach on several numerical examples from fluid dynamics (Section~\ref{sec:experiments}). For the same storage cost, the methods based on the $\starM$-product perform at least as well or better compared to the standard DMD. 
\end{enumerate}

\medskip
\paragraph{Related work} 
Since the initial paper on DMD by~\cite{schmid2010dynamic}, there have been several papers on DMD; see e.g., review articles~\cite{colbrook2023multiverse,  schmid2022dynamic}. There have been a few attempts to use tensors in the context of DMD. An approach for DMD using the TT-decomposition was proposed in~\cite{klus2018tensor}. In this approach, the snapshot matrices $\B{X}$ and $\B{Y}$ are reshaped into tensors and then compressed in the TT-format. In the formula $\B{YX}^\dagger$,  the pseudoinverse $\B{X}^\dagger$ is computed as the pseudoinverse of the matrix-unfolding. Two major differences compared to the present paper are: (1) the use of a different tensor format, and resulting tensor algebra, and (2) a different model for the dynamics. 
The recent papers~\cite{zhang2023multivariate,he2025tensor} introduced tensor DMD using the t-product (which is a special case of the $\starM$-framework). However, our paper goes beyond the extension to $\starM$-product in that it uses different compression schemes, explains the connections to standard DMD, and is adapted to the streaming setting. We expand on these points further.

Although our approach is tensor-based, we show that it has a close connection to traditional, matrix-based DMD and can be viewed from the lens of a regression viewpoint of DMD~\cite[Section 2.2.1]{colbrook2023multiverse}. We will also show that it has close connections with the physics informed DMD~\cite{baddoo2023physics}.

Our paper also addresses DMD in the streaming setting. This was first considered~\cite{hemati_dynamic_2014}. Another approach for updating the DMD was proposed in~\cite{zhang2019online}. But unlike our scheme, these approaches are not randomized.  Randomized matrix-based DMD algorithms were proposed in~\cite{ahmed_dynamic_2025,erichson2019randomized}, but to our knowledge they have not been applied to the streaming setting or for tensors. The present paper addresses these gaps.

\section{Background}\label{sec:background}

 \subsection{Linear Algebra background}

Let $\B{K} \in \C^{n\times n}$. We say that $\B{K} = \B{WTW}\t$ is a Schur form if $\B{W}$ is unitary and $\B{T}$ is upper triangular with the eigenvalues of $\B{K}$ as the diagonal elements of $\B{T}$. Even if $\B{K}$ is not be diagonalizable, the Schur form is guaranteed to exist. However, the form is not unique, since the diagonals can be chosen to appear in any order. Compared to the Jordan canonical form, it is more numerically stable to compute, since it employs unitary similarity transformations.   

Given a matrix $\B{A} \in \C^{m\times n}$, we denote the (full) SVD of $\B{A} = \B{U\Sigma V}^*$. If we truncate the SVD to a target rank $k \le \mathsf{rank}(\B{A})$, then we denote the truncated SVD $\B{A}\approx \B{U}_k\B{\Sigma}_k\B{V}_k^*$.

\subsection{Dynamic mode decomposition}\label{ssec:dmd} Consider a dynamical system represented by 
\[ \B{x}_{t+1} = \B{F}(\B{x}_t) \qquad 0 \le t < T,\]
where the states $\B{x}_t \in \Omega \subset \C^N $ for $0 \le t \le T$ and $\B{F}$ is the evolution operator. 
Dynamic mode decomposition posits that there is a matrix $\B{A}$ such that $\B{x}_{t+1} \approx  \B{Ax}_t$, and estimates $\B{A} \in \C^{N\times N}$ by solving the regression problem 
\begin{equation}\label{eqn:linreg}
    \min_{\B{A} \in \C^{N\times N}} \|\B{Y}-\B{AX}\|_F^2,
\end{equation}
where the snapshot matrices $\B{X} = \bmat{\B{x}_0 & \dots & \B{x}_{T-1}}$ and $\B{Y} = \bmat{\B{x}_1 & \dots & \B{x}_{T}}$. The minimum norm solution to this regression problem is given by $\B{A}_{\rm DMD} = \B{YX}^\dagger$ where $\B{X}^\dagger$ represents the Moore-Penrose inverse of $\B{X}$. In practice, a common approach is to truncate $\B{X}$ to rank-$r$ before applying the pseudoinverse. Assume that the target rank is $k \le \min\{N,T\}$ and we compute the approximation $\B{X} \approx \B{U}_k\B\Sigma_k\B{V}_k\t$. Then, we compute the matrix $\B{K} = \B{U}_k\t \B{Y}\B{V}_k\B\Sigma_k^{-1}$ and assuming it diagonalizable, we compute its eigendecomposition $\B{K} = \B{W\Lambda W}^{-1}$. We take the approximate eigenvectors of $\B{A}$ to be $\B\Phi = \B{U}_k\B{W}$ with corresponding eigenvalues $\B\Lambda$. Note that we have the approximation
\begin{equation}\label{eqn:galerkin} \B{A}_{\rm DMD} = \B{YX}^\dagger \approx \Bt{A}_{\rm DMD} \equiv \B{U}_k\B{U}_k\t \B{YX}^\dagger \B{U}_k\B{U}_k\t= \B{U}_k \B{KU}_k\t.  \end{equation}
Here $\Bt{A}_{\rm DMD}$ is an approximation to $\B{A}_{\rm DMD}$ obtained from Galerkin projection and $(\B\Lambda,\B\Phi)$ are the Ritz pairs. 
This process is outlined in Algorithm \ref{alg:exactdmd}.

\begin{algorithm}[!ht]
\caption{Exact DMD from \cite{schmid2010dynamic}}
\label{alg:exactdmd}
\begin{algorithmic}[1]
\REQUIRE Matrices $\B{X} = \bmat{\B{x}_0 & \dots & \B{x}_{T-1}}$ and $\B{Y} = \bmat{\B{x}_1 & \dots & \B{x}_{T}}$ in $\mathbb{C}^{N\times T}$; Target rank $k$; Assumption: $N>T$ and $k\leq \min\{N,T\}$.
\STATE $[\B{U},\B{\Sigma},\B{V}] = \mathsf{svd}(\B{X}, \mathsf{econ})$;
\STATE $\B{U}_{k} = \B{U}(:,1\!:\!k)$; $\B{\Sigma}_{k} = \B{\Sigma}(1\!:\!k,1\!:\!k)$; $\B{V}_{k} = \B{V}(:,1\!:\!k)$;
\STATE Compute $\B{K} = \B{U}_{k}^{*} \B{Y} \B{V}_{k} \B{\Sigma}_{k}^{-1}$;
\STATE $[\B{W}_{k},\B{\Lambda}_{k}] = \mathsf{eig}(\B{K})$;
\STATE $\B{\Phi}_{k} = \B{U}_{k}\B{W}_{k}$;
\RETURN $\B{\Phi}_{k} \in \mathbb{C}^{N\times k}$ (DMD modes) and $\B{\Lambda}_{k} \in \mathbb{C}^{k\times k}$ (DMD eigenvalues).
\end{algorithmic}
\end{algorithm}

To represent the states $\{\B{x}_t\}_{t=0}^T$, we consider the sequence of approximations 
\begin{equation*}
    \B{x}_{t} \approx \B{A}^t_{\rm DMD}\B{x}_0 \approx \Bt{A}_{\rm DMD}^t\B{x}_0, \qquad t \in \{0,\ldots,T \}.
\end{equation*}
The final expression $\Bt{A}^t\B{x}_0 = \B\Phi\B\Lambda^t \B\alpha$, where $\B\alpha \equiv \B\Phi^\dagger \B{x}_0 = \B{W}^{-1}\B{U}_k\t\B{x}_0$ follows from the properties of eigendecomposition. This choice of $\B\alpha$ is motivated by the solution to the least squares problem $\min_{\B\alpha \in \C^k}\|\B{x}_0- \B\Phi\B\Lambda^0\B\alpha\|_2^2 $. More generally, one can also obtain $\B\alpha$ by solving the least squares problem 
\[ \min_{\B\alpha \in \C^k} \sum_{j=0}^{T}\|\B{x}_j- \B\Phi\B\Lambda^j\B\alpha\|_2^2.\]
This is a structured least squares problem, for which there is a closed form expression for the optimal $\B\alpha$; see~\cite{jovanovic2014sparsity}. 

In practice, the matrix $\B{K}$ is not guaranteed to be diagonalizable. Even if it is diagonalizable, it may be nonnormal with a highly ill-conditioned eigenvector matrix $\B{W}$. In such cases, the eigendecomposition has numerical stability issues, which can be mitigated with by using the Schur decomposition~\cite{drmavc2023data}. We replace the eigendecomposition with the Schur decomposition, $\B{K} = \B{WTW}^*$ where $\B{W}$ is unitary and $\B{T}$ is upper-triangular, with the eigenvalues of $\B{K}$ as its diagonal entries. In our numerical implementation, we replace the eigendecomposition in step 4 of Algorithm~\ref{alg:exactdmd} with the Schur decomposition. 

\subsection{Tensor Notation}
We denote a third-order tensor as $\T{X} \in \C^{m \times p \times n}$ with entries $x_{ijk}$. There are three types of fibers: row fibers $\T{X}(i,:,k)$, column fibers $\T{X}(:,j,k)$, and tube fibers $\T{X}(i,j,:)$ for appropriate $i,j,k$. The frontal slices of $\T{X}$ are denoted $\T{X}(:,:,k)$ for $1 \le k \le n$, the lateral slices as $\T{X}(:,j,:)$ for $1\le j \le p$, and the horizontal slices as $\T{X}(i,:,:)$ for $1 \le i \le m$. The Frobenius norm of $\T{X}$ is $\|\T{X}\|_F = \left( \sum_{ijk} |x_{ijk}|^2 \right)^{1/2}$. 

The mode-$j$  unfolding of the third order tensor $\T{X}$ is denoted $\B{X}_{(j)}$ for $j=1,2,3$.   
These are $m \times (p \cdot n)$, $p \times (m \cdot n)$ and $n \times (p \cdot m)$ matrix reorganizations, respectively, of the data as `fat' matrices. 
In particular, $\B{X}_{(3)}$
is the $n \times (p\cdot m)$ matrix denoted as $\B{X}_{(3)}$ formed by slicing $\T{X}$ laterally into $m \times n$ matrices, and then transposing and stacking those matrices side-by-side, left to right.  
The mode-3 product between $\T{X}$ and a
$q \times n$ matrix $\B{A}$ will result in a tensor $\T{Y} \in 
\mathbb{C}^{m \times p \times q}$ and is expressed as $\T{Y} := \T{X}\times_3\B{A}$.  
It is computed according to $\B{Y}_{(3)}=\B{AX}_{(3)}$. 

Define an unfolding function $\unfold{\T{X}} = \B{X}_{(2)}\t \in \mathbb{C}^{mn \times p}$ that maps a tensor to a matrix and its inverse operation $\fold{\B{X}_{(2)}\t} = \T{X}$, that maps the the matrix to the corresponding tensor. Note $\unfold{\Th{X}} = (\B{M}\otimes \B{I})\unfold{\T{X}}$. Also, define the block diagonal matrix 
\[\bdiag{\T{X}} = \bmat{ \T{X}(:,:,1)\\ & \ddots\\ && \T{X}(:,:,n) } \in \C^{(mpn)\times(mpn) }. \]

\subsection{Star-M product}\label{ssec:starM}
The $\starM$-product framework was outlined in a sequence of papers \cite{kilmer2011factorization, kernfeld2015tensor,kilmer2021tensor}.  In this subsection, we review the definitions from those papers that are relevant to our tensor-based DMD formulation described in Section \ref{sec:tDMD}.

Given an invertible matrix $\B{M} \in \C^{n\times n}$, we can define a tensor $\T{X}$ in the so-called {\it transform domain} as $\Th{X} = \T{X}\times_3\B{M}$. Similarly, $\T{X} = \Th{X}\times_3\B{M}^{-1}$, and we refer to $\T{X}$ as living in the {\it standard domain}.   
With this in mind, the $\starM$ product of 
tensors is well-defined: 
\begin{definition}
Given $\T{A} \in \C^{m\times p \times n}$ and $\T{B} \in \C^{p\times s\times n}$, we can define the $\starM$-product $\T{C} = \T{A}\starM\T{B} \in \C^{m \times s\times n}$ by first computing the transform domain versions $\Th{A}$ and $\Th{B}$, and computing the frontal slice product through the standard matrix multiplication as 
\[ \Th{C}(:,:,k) = \Th{A}(:,:,k)\Th{B}(:,:,k)\qquad 1 \le k \le n. \]
Finally, $\T{C} = \Th{C}\times_3\B{M}^{-1}$.
\end{definition}

Many properties we list below extend to invertible $\B{M}$. However, from this point onward, we assume that $\B{M}$ is unitary.  This choice of course implies that the inverse is the conjugate transpose, making it easy/stable to move between the transform space and the standard domain.  

\begin{definition} If $\T{A} \in \C^{p \times p \times n}$, the identity tensor $p \times p \times n$ is defined through the operation 
$\T{A} \starM \T{I} = \T{A} = \T{I} \starM \T{A}$. It can be constructed in the transform domain as $\Th{I}(:,:,j) = \B{I}$ for $1 \le j \le n$ and $\T{I} = \Th{I} \times \T{M}^{-1}$. 
\end{definition}

\begin{definition} The conjugate transpose of $\T{A}$ is the tensor $\T{A}^*$, defined in the transform domain as $(\Th{A}^*)(:,:,j) = \Th{A}(:,:,j)^*$ for $1 \le j \le n$.  The tensor $\T{Q} \in \C^{m\times m \times n}$ is $\starM$-unitary if $\T{Q}\t\starM\T{Q} = \T{Q}\starM\T{Q}\t = \T{I}$. The tensor $\T{Q} \in \C^{m \times p \times n}$ has $\starM$-orthogonal slices if $\T{Q}\t \starM \T{Q} = \T{I}\in \C^{p \times p \times n}$.
\end{definition}

\begin{definition} 
Suppose $\T{A} \in \C^{m \times p \times n}$. The pseudo-inverse $\T{A}^\dagger \in \C^{p \times m\times n}$ is obtained by first converting $\T{A}$ to the transform domain $\Th{A} =\T{A}\times_3\B{M}$, and defining a tensor $\T{X}\in \C^{p \times m \times n}$ slicewise $\T{X}(:,:,i)  = \Th{A}(:,:,i)^\dagger$. Finally, we can define $\T{A}^\dagger = \T{X}\times_3 \B{M}^*$. 
\end{definition}

One of the best features of the $\starM$ product framework is that it admits a matrix-mimetic singular value decomposition:
\begin{theorem}[\cite{kilmer2021tensor}]
Let $\T{A}\in \mathbb{C}^{m \times p \times n}$ and let $\B{M}\in \mathbb{C}^{n \times n}$ be invertible. Then there exists a (full) $\starM$-SVD of $\T{A}$ given by
\begin{equation}
\T{A} = \T{U}\starM\T{S}\starM\T{V}\t = \sum_{i=1}^{r} \T{U}_{:,i,:}\starM\T{S}_{i,i,:}\starM\T{V}^{H}_{:,i,:},
\end{equation}
where $\T{U} \in \mathbb{C}^{m \times m \times n}$ and $\T{V} \in \mathbb{C}^{p \times p \times n}$ are $\starM$-unitary, $\T{S} \in \mathbb{C}^{m \times p \times n}$ has each of its frontal slices as diagonal (called f-diagonal), and $r\leq \min\{m,p\}$ is the number of nonzero tube fiberss in $\T{S}$ denoted $\B{s}_j$ for $1\le j \le \min\{m,p \}$. Moreover, $\|\B{s}_1\|_F^2 \ge \|\B{s}_2 \|_F^2 \cdots \ge \|\B{s}_{\min\{m,p\}}\|_F^2.$
\end{theorem}

The \textbf{multirank} of $\T{A}$ under $\starM$ is the vector $\B{\rho}$ such that its $i$-th entry $\rho_i$ denotes the rank of the $i$-th frontal slice of $\Th{A}$; that is, $\rho_i=\text{rank}(\Th{A}_{:,:,i})$. 
The number, $r$, of non-zero diagonal tubes $\B{s}_i$ in $\T{S}$ is the t-rank.  Note
that  $r = \max_{1 \le i \le n} \rho_i$.

In practice, it is more convenient to work with the thin-SVD of $\T{A}$  given by $\T{A} = \T{U}\starM\T{S}\starM\T{V}\t$, where now $\T{U} \in \C^{m \times q \times n}$, $\T{S}\in \C^{q \times q \times n}$ and $\T{V} \in \C^{p \times q \times n}$ and $q = \min\{m,p\}$. 

The fact that the Frobenius norm energy is non-increasing was used in \cite{kilmer2021tensor} to show that {\it truncating the $\starM$ SVD to $k < r$ terms gives an optimal approximation in the Frobenius norm}.  In the next section, we review the methods for truncating the $\starM$-SVD in an optimal way, as these truncation techniques are a key feature of our $\starM$-DMD approach.

\subsection{Algorithms} \label{ssec:algs}
Factorization of tensors under $\starM$ are defined and computed in the transform domain.  For example, it is possible to get a skinny QR factorization of a tensor by following Algorithm \ref{alg:QR}. 

\begin{algorithm}[!ht]
\caption{$\starM$-based Tensor QR}
\label{alg:QR}
\begin{algorithmic}[1]
\REQUIRE  $\T{A} \in \mathbb{C}^{m \times p \times n}$
\STATE $\widehat{\T{A}} = \T{A} \times_3 \B{M}$
\STATE $[\B{Q},\B{R}] = {\tt qr}(\widehat{\T{A}}(:,:,i),0)$ 
\STATE $\widehat{\T{Q}}(:,:,i)=\B{Q}; \widehat{\T{R}}(:,:,i)=\B{R}$;
\STATE $\T{Q} := \widehat{\T{Q}} \times_3 \B{M}^{*}$, $\T{R} := \widehat{\T{R}} \times_3 \B{M}^{*}$. 
\RETURN $\T{Q}, \T{R}$ such that $\T{A} = \T{Q} \starM \T{R}$.  
\end{algorithmic}
\end{algorithm}

Next, we briefly review two algorithms 
for truncated $\starM$-SVDs of the tensor $\T{A} \in \C^{m\times p \times n}$. 
Similar to the tensor QR in Algorithm \ref{alg:QR}, we operate on independent slices of $\Th{A}$.  The difference between the two methods has to do with the truncation procedure. 

\paragraph{\bf Truncated tSVDM Approximation (tr-tSVDM)} Suppose we compute a rank-$k$ approximation to the $j$th frontal slice of $\Th{A}$:
\[ \Th{A}(:,:,j) \approx \B{U}_{k}\B\Sigma_{k} \B{V}_{k}^*   \] 
for some truncation index $k \le \min(m,p)$. 
Define the tensors $\Th{U} \in \C^{m\times k \times n}$, $\Th{V}\in \C^{p \times k \times n}$, and $\Th{S} \in \C^{k\times k \times n}$ and populate their frontal slices as 
\[ \Th{U}(:,1\!:\!k,j)= \B{U}_k, \qquad
 \Th{S}(1\!:\!k,1\!:\!k,j)=\B\Sigma_k, \qquad
 \Th{V}(:,1\!:\!k,j)=\B{V}_k, \qquad 1\le j \le n.\]
Then, we convert back to the standard domain via the operations $\T{U}_k = \Th{U} \times_3 \B{M}^{*},$ $\T{S}_k = \Th{S} \times_3 \B{M}^{*},$ $\T{V}_k = \Th{V} \times_3 \B{M}^{*}$.  This gives the truncated $\starM$-SVD approximation $\T{A} \approx \T{A}_k = \T{U}_k \starM \T{S}_k \starM \T{V}_k^*, $ where the latter has t-rank $k$.  
By the Eckart-Young theorem~\cite[Theorem 3.8]{kilmer2021tensor}, this is an optimal t-rank approximation to $\T{A}$ in the Frobenius norm.

\paragraph{\bf Truncated tSVDMII Approximation (tr-tSVDMII)} Alternatively, 
 we can allow for different values of truncations $k_j$ with the understanding that some of the frontal slices of $\widehat{\T{A}}$ might have singular values that are much larger than other faces. Here,
 \[ \Th{U}(:,1\!:\!k_j,j) = \B{U}_{k_j}, \quad \Th{S}(1\!:\!k_j,1\!:\!k_j,j) = \B\Sigma_{k_j}, \quad \Th{V}(:,1\!:\!k_j,j) = \B{V}_{k_j} \qquad 1 \le j \le n.\] 
In order to be able to write this decomposition in the standard domain, we define $k=\max_{j} k_j$. For any face $j$ with $k_j < k$, we imagine padding with zeros (e.g. $\Th{U}(:,k_j\!+\!1\!:\!k,:) = {\bf 0}_{m \times (k-k_j)}$, etc.). Then $\T{U}_k, \T{S}_k, \T{V}_k$ are determined as in tr-tSVDM.  Note that this zero padding is {\it only for the convenience} of being able to write the decomposition in the standard domain.  In practice, all computations are done in the transform domain, which respects the potentially different values of $k_j$ along the faces.   
 The idea is to 
{\bf globally order} the $\hat{\sigma_{i}}^j$.
We keep only the largest $k_j$ of each face such that  
the relative error satisfies $\|\T{A} -\T{A}_{\B{k}}\|_F\le \sqrt{1 - \gamma}\|\T{A}\|_F$. The details of this approach are given in~\cite[Algorithm 3]{kilmer2021tensor}.  This approach is called tSVDMII in \cite{kilmer2021tensor}.  For clarity, we refer to it as tr-tSVDMII here.  

\paragraph{Remark}  Throughout the remainder of the paper, we will always be working with either of these truncated $\starM$-SVD approximations.  Thus, we will drop the $k$ subscript for the remainder, for ease of notation.

\subsection{Computational Cost} \label{ssec:cc}
The computational costs of both methods---tr-tSVDM and tr-tSVDMII---are the same in the asymptotic flop count if we assume that the full matrix SVDs are first computed for each face, then truncated. That is, we count the cost for computing the (thin) $\starM$-SVD for both truncation methods as the dominant cost.  We need the following notation: $T_{\B{M}}$ is the cost of a matrix-vector product (matvec) with $\B{M}$ and $\B{M}^*$ (both are $n\times n$). If $\B{M}$ is the discrete Fourier transfor (DFT) matrix, for example, then because we can use a fast Fourier transfor (FFT) to apply it, $T_{\B{M}} = \mc{O}(n\log n)$ flops.  For a general, dense unitary matrix, the cost would be $\mc{O}(n^2)$ flops.

The cost of transforming an $m \times p \times n$ tensor to and from the transform domain is $mpT_{\B{M}}$ flops, since we apply $\B{M}$ or $\B{M}^*$ to $mp$ fibers. The cost of an economy SVD for each frontal slice is $\mc{O}(mp\min\{p,m\})$ flops. Since there are $n$ slices, the total cost is $\mc{O}(mnp\max\{p,m\})$ flops. The total cost of the $\starM$ SVD is 
\[ T_{\rm tSVDM} = \mc{O}(mnp\min\{p,m\}) + 2mpT_{\B{M}} \> \text{flops}.\]
In comparison, if the SVD is used on the matrix $\B{A} \in \C^{(mn)\times p}$ obtained by unfolding, then the cost of SVD is $\mc{O}(mnp\min\{mn, p\})$ flops. If $m > p$, then the costs of the $\starM$-SVD and matrix SVD are comparable. However, one advantage of the $\starM$-SVD approach is the computational benefits in a parallel environment. Assume that we have a shared memory setting and there are $q$ processors. Then, the computations involving converting to transform domain and back, are perfectly parallelizable and the overall cost is $T_{\rm tSVDM}/q$ flops.

 \subsection{Factored-form Conversion to \texorpdfstring{$\starM$}{}-SVD format}\label{ssec:conversion}
 Suppose a tensor $\T{A}$ is defined as $\T{A} = \T{B} \starM \T{C}$ where $\T{B} \in \C^{m \times k \times n}$ and $\T{C}\t \in \C^{p \times k \times n}$ where $k \ll \min\{m,p\}$, where only $\T{B}, \T{C}$ are known. The slicewise rank of $\Th{A}$ is at most $k$. We want to find a $\starM$-SVD of $\T{A}$ {\bf without} first explicitly forming $\T{A}$.  This technique will be useful in the streaming setting of our algorithm.
 
 To do this, we first compute thin-$\starM$-QR factorizations $\T{B} = \T{Q}_B \starM \T{R}_B $ and $\T{C} = \T{Q}_C \starM \T{R}_C$. Then, we form $\T{T} := \T{R}_B \starM \T{R}_C\t $ and compute its $\starM$-SVD $\T{T} = \T{U}_T \starM \T{S} \starM \T{V}_T\t$. Then we have the $\starM$-SVD of $\T{A}$ as 
 \[ \T{A} = \T{Q}_B \starM \T{R}_B \starM \T{R}_C \starM \T{Q}_C\t = \T{U} \starM \T{S} \starM \T{V}\t, \]
 where $\T{U} = \T{Q}_B \starM \T{U}_T$ and $\T{V} = \T{Q}_C \starM \T{V}_T$. If additional truncation is necessary, it can be accomplished by truncating based on the singular values of $\Th{S}$. 

 Ignoring the cost to move in and out of the transform domain, the computational cost  of this approach is $\mc{O}(n(mk^2 + pk^2 + k^3))$ flops. 
 
\section{\texorpdfstring{$\starM$}{} DMD approach}  \label{sec:tDMD} 
In this section, we give an outline of the proposed $\starM$-DMD approaches  
and we outline the storage and computational costs of the proposed algorithms.

\subsection{Tensor Regression}\label{ssec:tenreg}
Let the state vectors $\{\B{x}_t\}_{t=0}^T$ represent the states which form the data matrices $\B{X}$ and $\B{Y}$ in $\C^{N \times {T}}$.  We can reshape the data matrices $\B{X}$ and $\B{Y}$ into tensors $\T{X}$ and $\T{Y}$ respectively  in $\C^{m \times p \times n}$ depending on the number of spatial dimensions. For two spatial dimensions, $N= n_xn_y$, so we take $m = n_x$, $p = T$, and $n = n_y$. Also, each lateral slice of $\T{X}$, denoted $\TA{X}_t \in \C^{m \times 1 \times n}$, represents the reshaped state $\B{x}_{t-1}$ for $1 \le t \le T$ (similarly, $\T{Y}$). Note that $\unfold{\T{Y}} = \B{Y}$ and $\unfold{\T{X}} = \B{X}$. A possible extension to three spatial dimensions is discussed in Section~\ref{sec:conc}.

We now discuss the tensor-based approach using the $\starM$-product. In the exact $\starM$-DMD approach, we can then approximate the state $\TA{X}_{t+1} \approx \T{A}\starM\TA{X}_t$  for $0\le t< T$. The tensor $\T{A}\in \C^{m \times m\times n}$ can be estimated by solving the regression problem:
\begin{equation}\label{eqn:tensorreg}  \min_{\T{A} \in \C^{m\times m \times n}} \| \T{Y} - \T{A}\starM \T{X}\|_F^2.\end{equation}
 The resulting least squares solution is $\T{A}_{\rm StarMDMD} = \T{Y}\starM\T{X}^\dagger$, where $\T{X}^\dagger$ was defined in Section~\ref{sec:background}.  

\subsection{\texorpdfstring{$\starM-$}{}Schur} Similar to the Schur decomposition in the matrix case, we can define the Schur decomposition for a tensor $\T{K} \in \C^{k  \times k \times n}$using the $\starM$-product. We convert $\T{K} $ to the transform domain to get $\Th{K}$. Next, we compute the Schur decomposition of the frontal slices as $\Th{K}(:,:,i) = \Th{W}(:,:,i) \Th{T}(:,:,i)\Th{W}(:,:,i)\t$ for $1 \le i \le n$. Finally, we obtain $\T{W} = \Th{W} \times_3 \B{M}\t$ and $\T{T} = \Th{T}\times_3 \B{M}\t$. This gives the decomposition $\T{K} = \T{W}\starM \T{T}\starM \T{W}\t$, where $\T{W}$ is $\starM$-unitary and $\T{T}$ has upper-triangular frontal slices. The computation of the Schur decomposition requires $\mc{O}(n k^3) + T_{\B{M}}k^2$ flops. As with the $\starM$-SVD, the computations can be accelerated in a shared memory parallel computation.

\subsection{Tensor DMD}\label{ssec:starmdmd} We proceed as in the exact DMD. We compute the truncated $\starM$-SVD $\T{X} \approx \T{U}\star_{\B{M}} \T{S}\star_{\B{M}} \T{V}^*$. 
This can be computed using the tr-tSVDM (i.e. $k_i=k$) or tr-tSVDMII as described in Section~\ref{ssec:algs}.

Next, we compute the tensor $\T{K} = \T{U}^* \star_{\B{M}} \T{Y} \star_{\B{M}} \T{V} \star_{\B{M}} \T{S}^\dagger $ and its Schur decomposition $\T{K} = \T{W} \star_{\B{M}} \T{T} \star_{\B{M}} \T{W}^* $. 
Note that, described in the standard domain, $\T{K}$ is $k \times k \times n$, with $k=\max_{1 \le j\le n} k_j$ if the tr-tSVDMII is used.  In practice, if the tr-SVDMII was used, only the leading $k_j \times k_j$ submatrix of $\Th{K}(:,:,j)$ will be non-zero, resulting in $O(k_j)^3$ flops per face to find the $\starM$-Schur decomposition.

Finally, we have the approximation 
\[  \T{A}_{\rm StarMDMD} \approx \T{Z}\star_{\B{M}} \T{T} \star_{\B{M}}\T{Z}^*, \qquad \T{A}_{StarMDMD} \in \mathbb{C}^{m \times m \times p}, \]
where $\T{Z} = \T{U}\star_{\B{M}} \T{W}$ determines the $k$ (approximate) modes and  the facewise upper-triangular tensor $\T{T}$ determines the eigenvalues. The procedure is described in Algorithm~\ref{alg:starmdmd}. Note that the computations are written in the standard domain in $\starM$-product form for ease of notation. However, in practice, we convert to the transform domain once at the start of the computations, perform the bulk of the computations in the transform domain, and return to the standard domain at the end of the computations. This is reflected in the storage and cost analysis in Section~\ref{ssec:costs}. 

\begin{algorithm}[!ht]
\begin{algorithmic}[1]
    \REQUIRE Tensors $\T{X}$ and $\T{Y}$ of size $m \times p\times n$ and unitary matrix $\B{M} \in \C^{n\times n}$
    \STATE Compute a {\it truncated} $\starM$-SVD $\T{X} \approx \T{U}\star_{\B{M}} \T{S}\star_{\B{M}} \T{V}^*$ \COMMENT{Either using tr-tSVDM or tr-tSVDMII.}
    \STATE Compute $\T{K} = \T{U}^* \star_{\B{M}} \T{Y} \star_{\B{M}} \T{V} \star_{\B{M}} \T{S}^\dagger $
    \STATE Compute the $\starM$-Schur decomposition $\T{K} = \T{W} \star_{\B{M}} \T{T} \star_{\B{M}} \T{W}^* $
    \STATE Compute the modes $\T{Z} = \T{U} \ast_{\B{M}} \T{W}$
    \RETURN Modes $\T{Z}$ and tensor containing eigenvalues $\T{T}$
\end{algorithmic}
    \caption{Outline of Tensor $\starM$ DMD}
    \label{alg:starmdmd}
\end{algorithm}

\subsection{Representation of States} We now discuss how to approximate the states from the DMD modes and eigenvalues. Suppose $\T{T}^t = \T{T} \starM \dots \starM \T{T}$ is defined in the usual way, with $t$ is a positive integer. Then, we can approximate the states as $\TA{X}_{t} \approx \T{Z}\starM \T{T}^t\starM \TA{G}$ for $0 \le t \le T$, where $\TA{G} \in \C^{k\times 1 \times n}$ is a tensor consisting of a single lateral slice, to be determined. 

There are several ways to define $\TA{G}$. In this paper, we take 
$\TA{G} = \T{Z}^* \starM \TA{X}_0$,  which solves the optimization problem $$\min_{\TA{G}} \| \TA{X}_{0} - \T{Z}\starM  \TA{G}\|_F$$ for the given initial state. An alternative approach would be to solve 
\[ \min_{\TA{G}} \sum_{t=0}^{n-1}\| \TA{X}_{t} - \T{Z}\starM  \T{T}^t \starM \TA{G} \|_F^2,  \]
 optimizing the coefficients $\TA{G}$ to recover the entire trajectory, rather than just the initial state. However, we do not pursue this approach in this paper.

\subsection{Storage and computational costs}\label{ssec:costs}
In $\starM$-DMD, there are two options for 
producing a truncated $\starM$-SVD approximation, so we will examine the storage costs and computational costs for each separately.  
For comparison, we use the exact-DMD but with the Schur decomposition instead of an eigendecomposition.

We first discuss the storage costs of the DMD modes, needed to reproduce the dynamics. If $\B{M}$ can be applied via fast transform (e.g. a DCT or FFT), we need not store the matrix explicitly.  In the case that moving in and out of the transform domain cannot be done via fast transform, one must also account for the cost of storing $\B{M}$, denoted $\mathsf{st}(\B{M})$, since we need $\B{M}$ to move in and out of the transform domain. 
Thus, if fast transforms are used, $\mathsf{st}(\B{M}) = 0$ whereas 
if $\B{M}$ is stored elementwise, we have $\mathsf{st}(\B{M}) = \mathsf{nnz}(\B{M})$, the number of non-zero elements in $\B{M}$.

If tr-tSVDM is used for approximating $\T{X}$, then the cost of  storing the DMD modes $\T{Z}$ are $mnk$ floating point numbers (flns), the cost of storing $\TA{G}$ is $kn$ flns and the cost of storing $\T{T}$ is $nk(k+1)/2$ flns.  We refer to this as $\starM$-DMD in Table \ref{tab:costs}.  If tr-tSVDMII is used, then it is more convenient to store this tensor implicitly in the transform domain.   Thus, the  cost to store $\T{Z}$, $\T{T}$, and $\TA{G}$ therefore is $m\sum_{j=1}^nk_j  + \sum_{j=1}^n [k_j(k_j+1)/2 + k_j] +   \mathsf{st}(\B{M})$ flns.  We refer to this as $\starM$-DMDII in Table \ref{tab:costs}.

\begin{table}[!ht]  \label{tab:costs}
    \centering
    \begin{tabular}{c|c}
          & Storage costs (flns)\\ \hline
      DMD   &   $mnk + k(k+1)/2$   \\ 
      $\starM$-DMD &  $mnk + nk(k+1)/2 + nk + \mathsf{st}(\B{M}) $   \\ 
      $\starM$-DMDII & $m\sum_{j=1}^nk_j  + \sum_{j=1}^n [k_j(k_j+1)/2  + k_j] + \mathsf{st}(\B{M})$  \\ \hline
       &  Computational cost (flops) \\
       \hline
      DMD   &  $\mc{O}(mnp\min\{nm,p\} + nk^2 +k^3 + mnk^2)$    \\ 
      $\starM$-DMD & $\mc{O}(mnp\min\{m,p\} + nk^3 + mnk^2) +  2mnpT_{\B{M}}$    \\ 
      $\starM$-DMDII & $\mc{O}(mnp\min\{m,p\} + \sum_{j=1}^nk_j^3 + m\sum_{j=1}^nk_j^2) +  2mnpT_{\B{M}}$
    \end{tabular}
    \caption{  Comparison of storage and computational costs of the different methods. }
   
\end{table}

Now we need to assess the computation cost beyond the cost of computing the $\starM$-SVD. Note that we only need to consider the cost of moving the tensor to the transform domain and back once for the $\starM$-DMD computations. First, we consider using tr-tSVDM. The cost of the $\starM$-Schur factorization is  $\mc{O}(k^3n)$ flops, and the cost to form $\T{Z}$ is $\mc{O}(nmk^2)$ flops. Next, considering the use of tr-tSVDMII, the cost of $\starM-$Schur factorization is  $\mc{O}(n\sum_{j=1}^nk_j^3)$ flops, and the cost to form $\T{Z}$ is $\mc{O}(m\sum_{j=1}^nk_j^2)$ flops.  The total cost of the two $\starM$-DMD algorithms, which include the cost of the $\starM$-SVD from Section \ref{ssec:cc}, is in Table \ref{tab:costs}. We compare to the total cost of the four steps of the standard matrix-based DMD in the table as well.

\section{Connections to Other DMD Frameworks} \label{sec:connections}
 
In this section, we explore the connection to the standard DMD (Section~\ref{ssec:connection}) and physics-informed DMD (Section~\ref{ssec:pidmd}). We give an interpretation of the modes in Section~\ref{ssec:interp}, explore an alternative optimization in Section~\ref{ssec:altopt}, and the Rayleigh-Ritz viewpoint in Section~\ref{ssec:rrview}. 

\subsection{Connection to Standard DMD}\label{ssec:connection}
We now explain how the $\starM$-DMD approximation is connected to the standard DMD approach. We will need a result about the $\starM$ product and a definition.

In the following result, we give an alternative way to compute the $\starM$ product. This result is previously known for the t-product (a special case of the $\starM$ product), but the result we present here is the general form, applicable for any choice of $\B{M}$. 
\begin{proposition}\label{prop:starmunfold} Let $\T{C} = \T{A} \starM \T{B}$ denote the $\starM$ product with $\B{M}$ invertible. Then, an alternate formula for $\T{C}$ is 
    \[ \begin{aligned}
        \T{C} = & \> \fold{(\B{M}^{-1}\otimes \B{I})\bdiag{\Th{A}} \unfold{\Th{B}}} \\
        = & \>\fold{(\B{M}^{-1}\otimes \B{I})\bdiag{\T{A}\times_3\B{M} } \unfold{\T{B} \times_3\B{M}}}.
    \end{aligned}  \]
\end{proposition}
\begin{proof}
    Follows from $\unfold{\Th{C}} =  \bdiag{\Th{A}} \unfold{\Th{B}}$ and $\unfold{\T{C}} =(\B{M}^{-1}\otimes \B{I})\unfold{\Th{C}}. $
\end{proof}
We can also compute $\T{C}$ through the operator 
\[ \unfold{\T{C}} = (\B{M}^{-1}\otimes \B{I}) \bdiag{\T{A}\times_3\B{M} }  (\B{M}\otimes \B{I})\unfold{\T{B}}.\]

 We will also need the following definition. 
 \begin{definition}
     Given  $\B{M}\in \C^{n_2\times n_2}$ unitary,  define 
 \[  \mb{A}^{n_1 \times n_3}_{\B{M}} \equiv \left\{ (\B{M}^{*}\otimes \B{I})\bdiag{\T{X}\times_3\B{M}}  (\B{M}\otimes \B{I}) \left|   \T{X} \in \C^{n_1\times n_3\times n_2} \right. \right\} .\] 
 This is a subspace of block-structured matrices in $\C^{(n_1n_2)\times (n_2n_3)}$ of dimension $n_1n_2n_3$.
 \end{definition} Note that:

\begin{enumerate}
    \item $\B{M}= \B{I}$, then $\mb{A}^{n_1 \times n_3}_{\B{I}}$ is the subspace of block diagonal matrices 
    \item $\B{M} = \B{F}$ (DFT matrix), then $\mb{A}^{n_1 \times n_3}_{\B{F}}$ is the subspace of block circulant matrices. 
\end{enumerate}

 We now show the connection between $\starM$ DMD and the standard DMD. For any $\T{A} \in \C^{m \times m \times n}$ and $\B{M} \in \C^{n\times n}$, we can rewrite 
\[ \| \T{Y} - \T{A}\ast_{\B{M}}\T{X}\|_F^2 = \|\B{Y} - (\B{M}^{*}\otimes \B{I})\bdiag{\T{A}\times_3\B{M}}  (\B{M}\otimes \B{I})\B{X}\|_F^2.  \] 
This means that we can rewrite~\eqref{eqn:tensorreg} as the least squares problem as
\begin{equation} \label{eqn:dmdmatlstsq}\min_{\B{A} \in \mb{A}_{\B{M}}^{m \times m}} \| \B{Y} - \B{AX}\|_F^2.  \end{equation}
In other words, compared to~\eqref{eqn:linreg}, where the optimization is unconstrained,~\eqref{eqn:tensorreg} solves a highly structured constrained problem. The optimal solution is 
\begin{equation}\label{eqn:starmdmd}\B{A}_{\rm StarMDMD} = (\B{M}^*\otimes \B{I})\bdiag{\T{A}_{\rm StarMDMD}\times_3\B{M}}  (\B{M}\otimes \B{I}).\end{equation}  

\subsection{Connection to Physics-informed DMD}\label{ssec:pidmd} The optimization problem in \eqref{eqn:dmdmatlstsq} is closely related to the Physics-informed DMD approach (piDMD)~\cite{baddoo2023physics}. In piDMD, the optimization problem takes the form 
\[ \min_{\B{A} \in \mc{M}} \| \B{Y} - \B{AX}\|_F^2,\] 
where $\mc{M}$ is an appropriate subset of $\C^{mn \times mn}$. For example, one option is to impose the constraint that $\B{A}$ is unitary, i.e., $\mc{M}$ is the unitary group,  which represents a rotation or reflection. In another option, $\mc{M}$ is the subspace of circulant matrices, which preserves shift-invariance. Other options for $\mc{M}$ include the space of Hermitian, skew-Hermitian, banded, and upper triangular. 

In contrast, in the starMDMD, the constraint is $\mc{M} = \mb{A}^{m\times m}_{\B{M}}$ is defined by the $\starM$-product. As mentioned earlier, if $\B{M}=\B{I}$, then $\mb{A}^{m\times m}_{\B{M}}$ is the space of block-diagonal matrices, and if $\B{M}$ is the DFT matrix, then $\mb{A}^{m\times m}_{\B{M}}$ is the space of block-circulant matrices. This latter option corresponds to the shift-equivariant case~\cite{baddoo2023physics}.

\subsection{Interpretation of the Modes}\label{ssec:interp}
If no truncation is used in the exact $\starM$-DMD approach, then we can write $\T{A}_{\rm StarMDMD} = \T{Z} \starM \T{T} \starM \T{Z}\t$ where $\T{Z}$ is unitary, and $\T{T}$ has upper triangular frontal slices in the transform domain. Using Proposition~\ref{prop:starmunfold}, we can express  $\T{A}_{\rm StarMDMD}$ in  matrix form as 
\[ \B{A}_{\rm StarMDMD} =  (\B{M}\t\otimes \B{I}) \bdiag{\Th{Z}} \bdiag{\Th{T}} (\bdiag{\Th{Z}})\t (\B{M}\otimes \B{I}),   \] 
where hats denote tensors in  the transform domain. This represents a Schur factorization of the form $\B{A}_{\rm StarMDMD} = \B{WTW}^*$, where $\B{W} = (\B{M}\t\otimes \B{I}) \bdiag{\Th{Z}}$ is unitary and $\B{T} = \bdiag{\Th{T}}$ is upper triangular; more specifically, it is (block) upper triangular, with diagonal blocks in upper triangular form. Therefore, the columns of $\B{W}$ represent the Schur vectors, and the diagonals of $\B{T}$ represent the eigenvalues of $\B{A}_{\rm StarMDMD}$. Note that Schur vectors are not guaranteed to be eigenvectors. 

\subsection{Alternative optimization}\label{ssec:altopt} Instead of \eqref{eqn:tensorreg}, we can solve an alternative optimization problem 
\begin{equation}
    \label{eqn:altopt}
    \min_{\B{A} \in \mb{A}_{\B{M}}^{m \times m}} \|\B{YX}^\dagger - \B{A}\|_F^2.
\end{equation}   
That is, we seek the closest approximation in $\mb{A}_{\B{M}}^{m \times m}$ to the exact matrix DMD. However, the minimizer of this optimization problem is not $\B{A}_{\rm StarMDMD}$, in general. 

\begin{proposition}\label{prop:altopt}
Let  $\T{G}$ be a tensor constructed such that $\T{G} = \T{Y} \starM \T{Z}$, where $\Th{Z}(:,:,i) =  [\unfold{\Th{X}}^\dagger] (:,(i-1)*p+1:i*p )$. Then 
\[ (\B{M}^{*}\otimes \B{I})\bdiag{\T{G}\times_3\B{M}}  (\B{M}\otimes \B{I}),\]
is a minimizer of~\eqref{eqn:altopt}. 
\end{proposition}
\begin{proof}
Let $\B{A} = (\B{M}^{*}\otimes \B{I})\bdiag{\T{A}\times_3\B{M}}  (\B{M}\otimes \B{I})$. By unitary invariance of the Frobenius norm and~\cite[Theorem 2.2.3]{bjorck2015numerical}
\[ \begin{aligned}
    \|\B{YX}^\dagger - \B{A}\|_F = & \>  \| (\B{M}\otimes \B{I}) \B{YX}^\dagger (\B{M}\t\otimes \B{I}) -  \bdiag{\Th{A}}  \|_F\\
    = & \| \unfold{\Th{Y}} [\unfold{\Th{X}}]^\dagger-  \bdiag{\Th{A}}\|_F^2. 
\end{aligned} \] 
Because of the block diagonal structure, the minimum norm solution $\Th{G}(:,:,i) = \Th{Y}(:,:,i) \Th{Z}(:,:,i) $ for $1\le i \le n$. The rest follows from properties of $\starM$ product. 
\end{proof}
Note that $\T{Y}\starM \T{X}^\dagger$ is different from $\T{G}$ as defined in Proposition~\ref{prop:altopt}, since the slicewise pseudoinverse differs from computing slices of the pseudoinverse. A disadvantage of this approach is that the step of forming $\unfold{\Th{X}}^\dagger$ is not readily parallelizable across the slices, so the approach by solving~\eqref{eqn:tensorreg} is preferable from that point of view. Therefore, we did not explore this approach further.  
\subsection{Rayleigh-Ritz Viewpoint}\label{ssec:rrview}
The exact DMD is related to the Rayleigh-Ritz approach for estimating eigenpairs, which we briefly review. Consider the eigenvalue problem $\B{Ax} = \lambda \B{x}$. Given a subspace spanned by the columns of the matrix $\B{V} \subset \C^k$, with orthonormal columns, we can consider the Rayleigh-Ritz approximation
\[  \B{Ax} - \lambda\B{x} \perp \range(\B{V}), \qquad \B{x} \in \range(\B{V}),\] 
where $\range(\B{V})$ is the range of the matrix, or the span of the columns of $\B{V}$.  Thus, in the Rayleigh-Ritz approach, we search for solutions $\B{x} \in \range(\B{V})$ and obtain them by setting the residual $\B{Ax} -\lambda \B{x}$ orthogonal to $\range(\B{V})$. The solutions can be obtained by forming $\B{B} = \B{V}\t\B{AV}$ and computing its eigenpairs $(\theta_i,\B{y}_i)$ for $1 \le i \le k$. Finally, we can obtain the approximate eigenpair (or Ritz pairs) as $(\theta_i,\B{Vy}_i)$ for $1 \le i \le k$. It can be seen that the Ritz pairs are the eigenpairs of $\Bt{A} = \B{VV}\t\B{AVV}\t$. Furthermore, we can use the Schur vectors instead of the eigenvectors. See~\cite[Section 4.3.1]{saad2011numerical} for more details. 

With this viewpoint,~\eqref{eqn:galerkin} shows that the eigenpairs of $\Bt{A}_{\rm DMD}$ are the Ritz pairs corresponding to $\B{A}_{\rm DMD}$. We can extend this interpretation to the $\starM$-product. Consider $\B{A}_{\rm starMDMD}$ as in~\eqref{eqn:starmdmd} formed from the tensor $\T{A}_{\rm StarMDMD}$. We then form the approximation $\Tt{A}_{\rm StarMDMD} = \T{U}\starM \T{K} \starM \T{U}\t$, where $\T{U}$ is obtained from the truncated $\starM$-SVD with multirank $\B{k}$ and $\T{K} = \T{U}\t \starM \T{Y}\starM \T{X}^\dagger \starM \T{U}$. Analogous to the derivation of $\B{A}_{\rm StarMDMD}$, we can show that $\Tt{A}_{\rm StarMDMD}$ corresponds to the matrix  
\[ \Bt{A}_{\rm StarMDMD} = \B{P} \B{A}_{\rm StarMDMD} \B{P},  \]
where $\B{P} = \B{VV}\t$ is an orthogonal projector with  $\B{V} = (\B{M}\t\otimes \B{I})\bdiag{\Th{U}}$. Thus, the Ritz pairs of $\B{A}_{\rm StarMDMD}$ corresponding to the subspace spanned by the columns of $(\B{M}\t\otimes \B{I})\bdiag{\Th{U}}$ are the eigenpairs of $\Bt{A}_{\rm StarMDMD}$.

\section{Streaming \texorpdfstring{$\starM$}{}-DMD}\label{sec:streaming}
In this section, we derive a randomized algorithm for $\starM$-DMD in the streaming setting.  Note that when $n=1$, the tensor collapses to a matrix. Then, taking $\B{M}=1$, 
all tensor operations collapse to matrix operations.  Thus, our description covers both the streaming model in the matrix AND tensor cases.

First, we clarify the streaming model in Section~\ref{ssec:streaming}. Next, we address $\starM$-SVD in the streaming setting (Section~\ref{ssec:starmstreamingsvd}), before discussing its implementation inside the $\starM$-DMD (Sections~\ref{ssec:starmstreamingdmd}).

\subsection{Streaming setting}\label{ssec:streaming} In this setting, we assume that the lateral slices $\TA{C}_t$ arrive sequentially, in batches. Once the batch arrives, we process it and then discard the batch. We are no longer allowed access to the batches.  Suppose there are $b$ different batches of indices $B_1 = \{0,\dots,i_1\}$, $B_2 = \{i_1+1,\dots,i_2\}, \dots, B_b =\{i_{b-1}+1,\dots,T\}$, where $ 0 < i_1 < i_2 < \dots < i_{b-1}$. Note that we can decompose the tensor $\T{C}$ that contains all the snapshots as lateral slices into $b$ tensors 
\[ \T{C} = \T{C}_1 + \dots + \T{C}_b ,\] 
where $\T{C}_j$ contain the lateral slices indexed by $B_j$ and zeros otherwise. 

Rather than recomputing the DMD modes from scratch each time, the goal is to update the decomposition sequentially. We will exploit the linearity of the tensor in the streaming approach. First, we discuss how to obtain the compressed representation of $\T{C}$.

\subsection{Streaming \texorpdfstring{$\starM$}{}-SVD}\label{ssec:starmstreamingsvd}
We follow the randomized single-view algorithm for matrices proposed in~\cite{tropp2017practical}. In this approach, we draw two random tensors $\T{G}_1 \in \R^{p \times \rho_1 \times n}$ and $\T{G}_2 \in \R^{ \rho_2 \times m\times n}$. To specify the distribution of the entries, we will need the following definition. 
\begin{definition}[Random Gaussian tensor]\label{def:randgauss}
    A random Gaussian tensor $\T{G}\in \C^{m \times p \times n}$ is defined as follows. Let $\B{G} \in \R^{m \times p}$ be a standard Gaussian random matrix (independent entries drawn from normal distribution with zero mean and unit variance). Define the tensor in the transform domain as $\Th{G}(:,:,j) = \B{G}$ for $1 \le j \le n$; finally, $\T{G} = \Th{G} \times \B{M}^*$.  
\end{definition}

We compute the sketches $\T{Y}_1 = \T{C}\starM \T{G}_1$ and $\T{Y}_2 = \T{G}_2 \starM \T{C}$. Note that the sketches $\T{Y}_1$ and $\T{Y}_2$ can be updated sequentially by the linearity of $\T{C}$.

We compute a thin-QR factorization of $\T{Y}_1 = \T{Q}_1 \starM \T{R}_{1}$. Then, we obtain the low-rank approximation 
\begin{equation}\label{eqn:streamingapprox} \T{C} \approx \Tt{C} = \T{Q}_1 \starM (\T{G}_2\t \starM \T{Q}_1)^\dagger  \starM \T{Y}_2.  \end{equation}
The following result quantifies the error in the low-rank approximation, and follows from the results in~\cite{tropp2017practical} and the properties of the $\starM$-product. We need the following notation. Let $k$ be the target rank for each frontal slice. Let $\T{C}_{(k)}$ be the $\starM$-SVDM approximation to $\T{C}$. Then,
\[ \|\T{C} - \T{C}_{(k)}\|_F^2= \sum_{j=1}^n \sum_{i> k} (\hat\sigma^{(j)}_i)^2,\]
where $\hat\sigma^{(j)}_i$ denote the $i$-th singular value (here $1 \le i \le \min\{m,p\}$) of the $j$-th frontal slice (for $1 \le j \le n$) in the transform domain. 

\begin{theorem} Let $\T{G}_1 \in \C^{p \times \rho_1 \times n} $ and $\T{G}_2 \in \C^{\rho_2 \times m \times n}$ be two independent Gaussian tensors (see Definition~\ref{def:randgauss}), constructed using standard Gaussian random matrices $\B{G}_1 \in \R^{p \times \rho_1} $ and $\B{G}_2\in \R^{\rho_2 \times m}$, respectively, as the frontal slices in the transform domain. Construct the approximation $\Tt{C}$  to $\T{C}$ using~\eqref{eqn:streamingapprox}. Then, with $\rho_1 = 2k + 1$ and $\rho_2 = 2 \rho_1 + 1$,   
\[ \mb{E}_{\B{G}_1, \B{G}_2}[ \| \T{C} - \Tt{C}\|_F^2  ]  \le 4 \| \T{C} - \T{C}_{(k)}\|_F^2.  \]
    
\end{theorem}
\begin{proof}
    By the unitary invariance of the Frobenius norm
    \begin{equation}\label{eqn:unitinvar}\| \T{C} - \Tt{C}\|_F^2 = \sum_{j=1}^n \| \Th{C}(:,:,j) - \Th{Q}_1(:,:,j) ( \Th{G}_2(:,:,j) \Th{Q}_1(:,:,j) )^\dagger \Th{Y}_2(:,:,j)\|_F^2. \end{equation}
    Apply~\cite[Theorem 4.3]{tropp2017practical}, for each frontal slice $1 \le j \le n$
    \[ \mb{E}_{\B{G}_1,\B{G}_2} \left[\| \Th{C}(:,:,j) - \Th{Q}_1(:,:,j) ( \Th{G}_2(:,:,j) \Th{Q}_1(:,:,j) )^\dagger \Th{Y}_2(:,:,j)\|_F^2\right] \le  4 \sum_{i> k} (\hat\sigma^{(j)}_i)^2. \] 
    The proof is complete by taking expectations in~\eqref{eqn:unitinvar}, using the linearity of expectations,  and the above inequality.
\end{proof}
The interpretation of this result is straightforward. By the Eckart-Young theorem for $\starM$-product~\cite[Theorem 3.8]{kilmer2021tensor}, $\T{C}_{(k)}$ is an optimal approximation to $\T{C}$ with multirank $\B{k} = (k,\dots,k)$. In expectation, the streaming approximation defined in~\eqref{eqn:streamingapprox} is within a factor of $4$ of the best possible error.

The low-rank approximation $\Tt{C}$ can be postprocessed to convert it into the $\starM$-SVD format using the technique in Section~\ref{ssec:conversion}. In addition, we truncate the tensor using tr-tSVDMII on $\Tt{C}$ with a parameter $\gamma \in (0,1]$ to obtain the final low-rank approximation $\Tt{C}_\gamma$.

\subsection{Streaming \texorpdfstring{$\starM$}{}-DMD}\label{ssec:starmstreamingdmd}
At the end of the streaming approach, we have a low-rank approximation to $\T{C}$ in the $\starM$-format. Before we apply the $\starM$-DMD approach, there are two differences to which we must pay attention. First, in the $\starM$ DMD, we have two tensors $\T{X}$ and $\T{Y}$ containing the snapshots $\{\B{x}\}_{j=0}^T$ as lateral slices. The tensor $\T{X}$ is obtained from $\T{C}$ by deleting the first column, and $\T{Y}$ is obtained from $\T{C}$ by deleting the last.  Second, we must work with the low-rank approximation to $\T{C}$ rather than $\T{X}$ directly. One can form an explicit approximation to $\T{X}$ from the low-rank approximation and then invoke Algorithm~\ref{alg:starmdmd}. But, we favor a slightly different approach that is more computationally efficient.

In this approach, we assume that we have the low-rank approximation in the $\starM$-format $\T{C} \approx \T{U}\starM \T{S}\starM \T{V}\t$. Next, we delete the first lateral slice from $\T{V}\t$ and the last lateral slice to obtain the tensors $\T{B}'$ and $\T{B}''$ respectively. Thus, we have the approximations 
\[ \T{X} \approx \T{U}\starM \T{S} \starM \T{B}' \qquad \T{Y} \approx \T{U} \starM \T{S} \starM \T{B}''. \]
From this point, we convert the low-rank approximation to the $\starM$-SVD format as $\T{U}\starM \T{S} \starM \T{B}' = \T{U}' \starM \T{S}' \starM (\T{V}')\t$; see~Section~\ref{ssec:conversion}. We can then proceed with the $\starM$-DMD approach as in Section~\ref{ssec:starmdmd}, but with the low-rank approximations to $\T{X}$ and $\T{Y}$, rather than the tensors themselves.

\begin{algorithm}[!ht]
    
    \caption{$\starM$-DMD from low-rank approximation of states}
    \label{alg:starmdmdlowrank}
    \begin{algorithmic}[1]
        \REQUIRE Input tensor approximation $\T{C} \approx \T{U} \starM \T{S} \starM \T{V}\t$
        \STATE Compute tensors $\T{B}'$ and $\T{B}''$ by deleting the last and first lateral slice, respectively, of $\T{V}\t$ 
        \STATE Compute thin-$\starM$-SVD $\T{S} \starM \T{B}' = \Tt{U}\starM \T{S}' \starM (\T{V}')\t$ and $\T{U}' = \T{U} \starM \Tt{U}$
        \STATE Compute $\T{K} = (\T{U}')\t \starM \T{U} \starM \T{S} \starM \T{B}'' \starM \T{V}' \starM  (\T{S}')^\dagger $
        \STATE Compute $\starM$-Schur decomposition $\T{Z} = \T{W} \starM \T{T} \starM \T{W}\t$
        \STATE Compute the modes $\T{Z} = \T{U}'\ast_{\B{M}} \T{W}$
    \RETURN Modes $\T{Z}$ and tensor containing eigenvalues $\T{T}$
    \end{algorithmic}
\end{algorithm}

In Algorithm~\ref{alg:starMsteaming}, we give an overview of the steps of the algorithm in the streaming setting. This algorithm combines the $\starM$-SVD with the $\starM$-DMD approach described just previously. 

\begin{algorithm}[!ht]
    \caption{Streaming $\starM$-DMD}
    \label{alg:starMsteaming}
    \begin{algorithmic}[1]
        \REQUIRE Tensors $\{\T{C}_j\}_{j=1}^b$ such that $\T{C} = \T{C}_1 + \dots + \T{C}_b$, parameters $\rho_{\max}$ and $\gamma \in (0,1]$. 
        \STATE Select $\rho_1= k_{\max}$ and $\rho_2 = 2\rho_1 + 1$
        \STATE Generate random tensors $\T{G}_1 \in \C^{p \times \rho_1 \times n}$ and $\T{G}_2 \in \C^{ \rho_2 \times m\times n}$ (see Definition~\ref{def:randgauss})
        \STATE Initialize sketches $\T{Y}_1 \in \C^{m \times k_1 \times n}$ and $\T{Y}_2 \in \C^{k_2 \times p \times n}$ to have zero entries
        \FOR{$j = 1,\dots,b$}
        \STATE Update sketches $\T{Y}_1  \leftarrow  \T{Y}_1   + \T{C}_j \starM \T{G}_1$, $\T{Y}_2 \leftarrow \T{Y}_2 + \T{G}_2 \starM \T{C}_j$
        
        \STATE Compute thin-$\starM$-QR factorization $\T{Y}_1 = \T{Q}_1 \starM \T{R}_{11}$. 
        \STATE Compute $\T{B} =(\T{G}_2\starM \T{Q}_1)^\dagger \starM \T{Y}_2$ \COMMENT{The intermediate approximation is $\sum_{i=1}^j \T{C}_i \approx \T{Q}_1 \starM \T{B}$}
        \STATE Apply tr-tSVDMII to $\T{B}$ with parameter $\gamma$ to get the low-rank approximation $\T{U}_1 \starM \T{S} \starM \T{V}\t$
        \STATE Compute $\T{U} = \T{Q}_1 \starM \T{U}_1$
        \STATE Use Algorithm~\ref{alg:starmdmdlowrank} with inputs $\T{U}$, $\T{S}$, $\T{V}$ to compute modes $\T{Z}$ and tensor containing eigenvalues $\T{T}$
        \ENDFOR 
        \RETURN Tensor $\T{Z}$, containing $\starM$-DMD modes,  and tensor $\T{T}$ containing eigenvalues
    \end{algorithmic}
\end{algorithm}

We now present the computational and storage costs associated with Algorithm \ref{alg:starMsteaming}.  

\paragraph{Computational costs} We discuss the computational cost for the updating algorithm. For simplicity, we assume that the batch size remains constant, that is, $B_1 = \dots = B_b = B$ and we assume $\rho_1$ and $\rho_2$ are both $\mc{O}(\rho)$. The cost of updating the sketches are $\mc{O}(mnB) + 2T_{\B{M}}mB$ flops. The subsequent tr-tSVDM computations cost $\mc{O}( mn\rho^2 + pn\rho^2 )$ flops. After truncation, assume that the multirank is $(k_1,\dots,k_n)$. Therefore, the cost of the $\starM$-DMD computations is $\mc{O}((m+p)\sum_{j=1}^nk_j^2 + \sum_{j=1}^nk_j^3)$ flops. Finally, there is another cost of $T_{\B{M}}m\rho$ flops for bringing it back into the standard domain. As with the $\starM$-DMD, there is enormous opportunity for parallelism.

\paragraph{Storage costs}
In Algorithm~\ref{alg:starMsteaming}, we have to draw two random matrices that define the tensors $\T{G}_1$ and $\T{G}_2$. These have a storage cost of $p\rho_1 + m\rho_2$ flns. Next, we must store the sketches $\T{Y}_1$ and $\T{Y}_2$ which have a combined storage cost of $n(m\rho_1+p\rho_2)$ flns.

\section{Numerical experiments}\label{sec:experiments}
In this section, we assess the performance of the two proposed static $\starM$-DMD algorithms on a cylinder flow dataset \cite[Chapter 2]{kutz2016dynamic} and a sea surface temperature dataset \cite{reynolds2002improved}. Additionally, we compare the performance of the streaming $\starM$-DMDII algorithm and the streaming DMD algorithm on the second cylinder flow dataset, namely the ``Cylinder Flow with von Karman Vortex Sheet'' dataset \cite{gerrisflowsolver, Guenther17}.

\paragraph{Choice of $\B{M}$} For the $\starM$-DMD and the $\starM$-DMDII methods, we have three different choices of the orthogonal matrix $\B{M}$: the discrete cosine transform (DCT) matrix, the normalized discrete sine transform (DST) matrix, and the left singular factor matrix of the transpose of the mode-3 unfolding of the given tensor. The last orthogonal matrix $\B{M}$ obtained is said to be from a data-driven approach, as discussed in \cite{kilmer2021tensor}. 

\paragraph{Storage cost comparison} When the orthogonal matrix $\B{M}$ is either the DCT matrix or the DST matrix, we can omit the storage costs of $\B{M}$ \cite{kilmer2021tensor}. In the case where $\B{M}$ is the data-driven orthogonal matrix, then the storage for $\B{M}$ is bounded by the number of nonzeros in $\B{M}$.  

\subsection{Cylinder flow} The cylinder flow dataset is obtained from a simulation of the two-dimensional Navier-Stokes equations, representing a time series of fluid
vorticity fields for the wake behind a circular cylinder at Reynolds number $\text{Re} = 100$. In total,  $150$ snapshots are collected, and each snapshot is of size $450\times 200$.

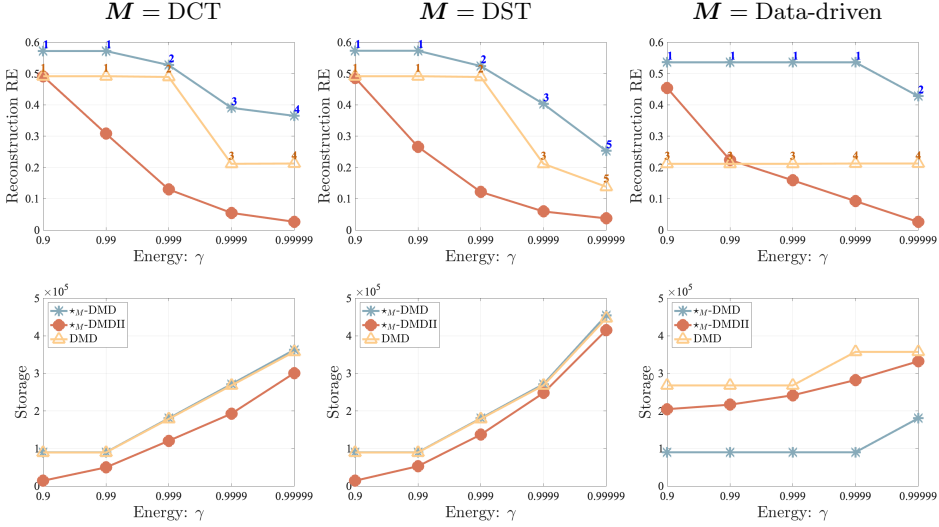
\begin{figure}[!ht]
\centering
\begin{tikzpicture}
\node at (0, 3.4) {\includegraphics[width=0.33\linewidth]{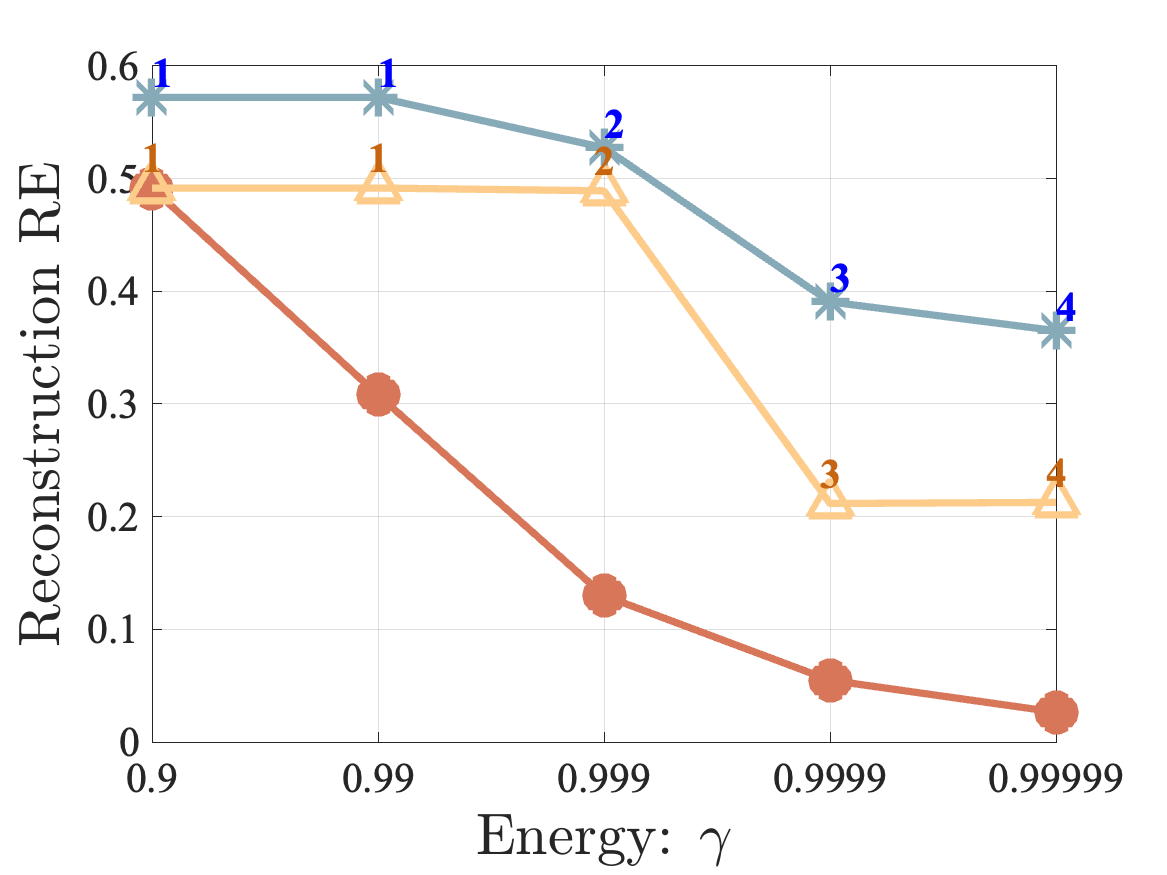}};
\node at (4.15, 3.4) {\includegraphics[width=0.33\linewidth]{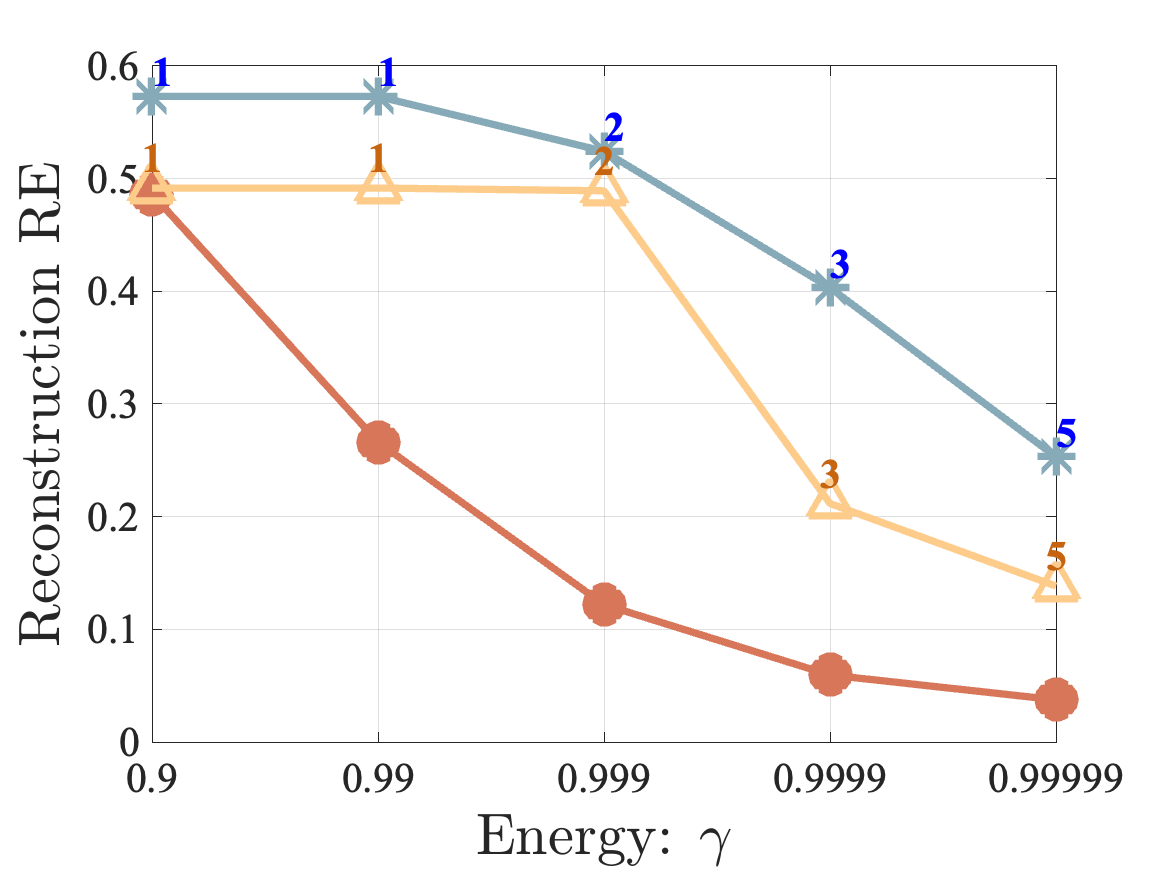}};
\node at (8.3, 3.4) {\includegraphics[width=0.33\linewidth]{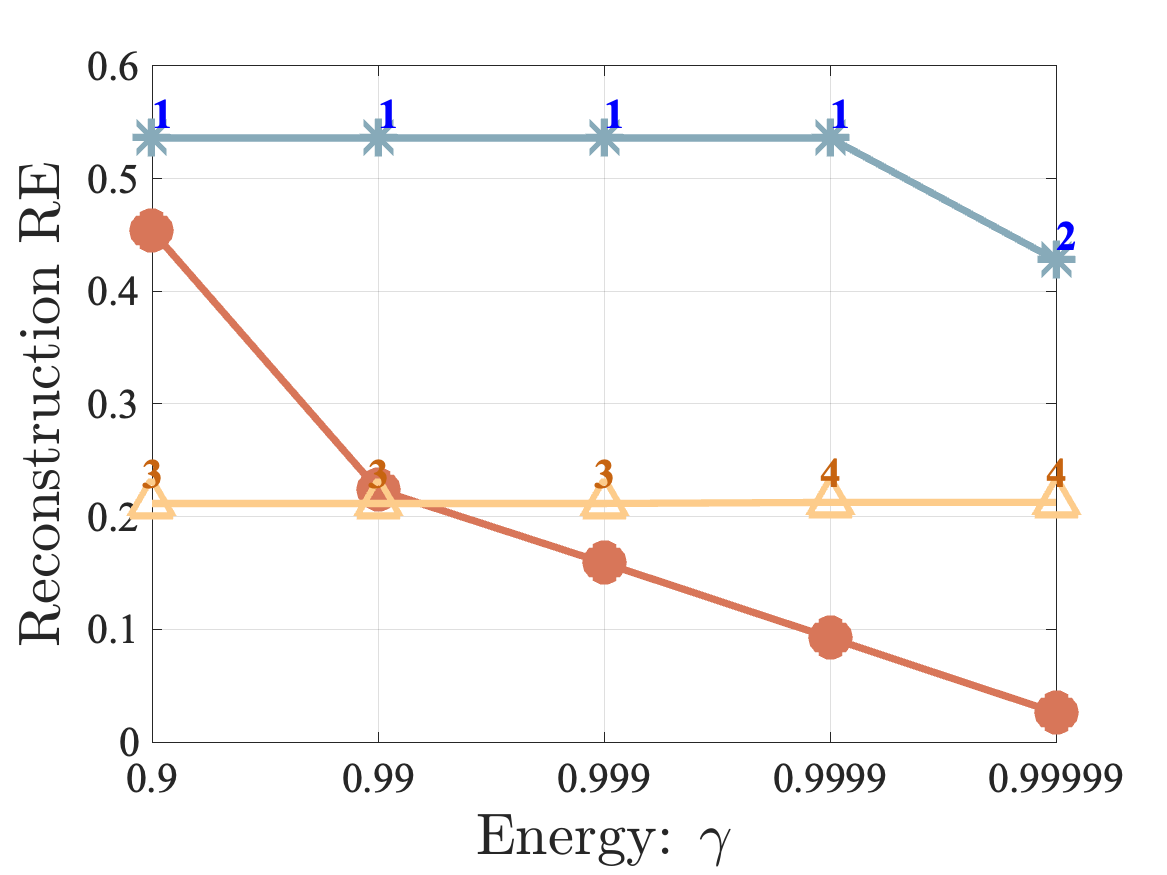}};

\node at (0, 0) {\includegraphics[width=0.33\linewidth]{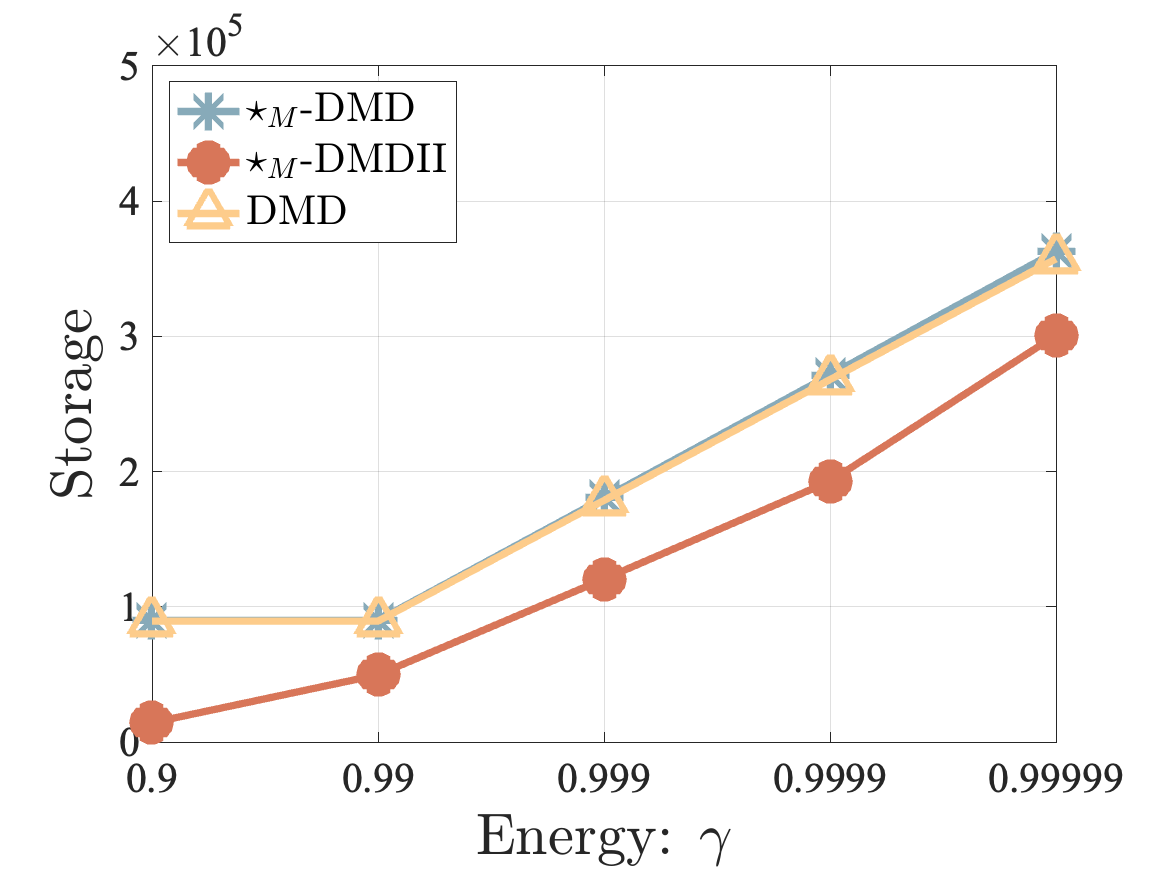}};
\node at (4.15, 0) {\includegraphics[width=0.33\linewidth]{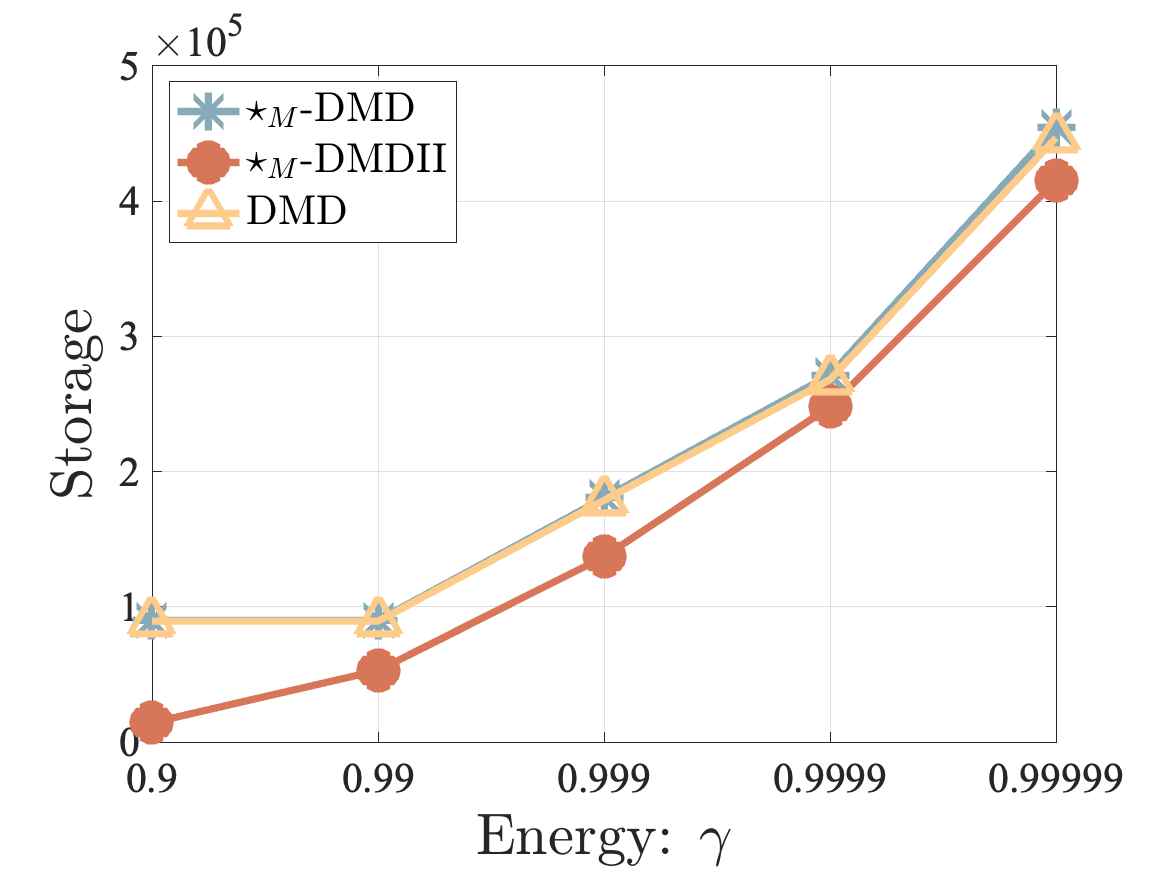}};
\node at (8.3, 0) {\includegraphics[width=0.33\linewidth]{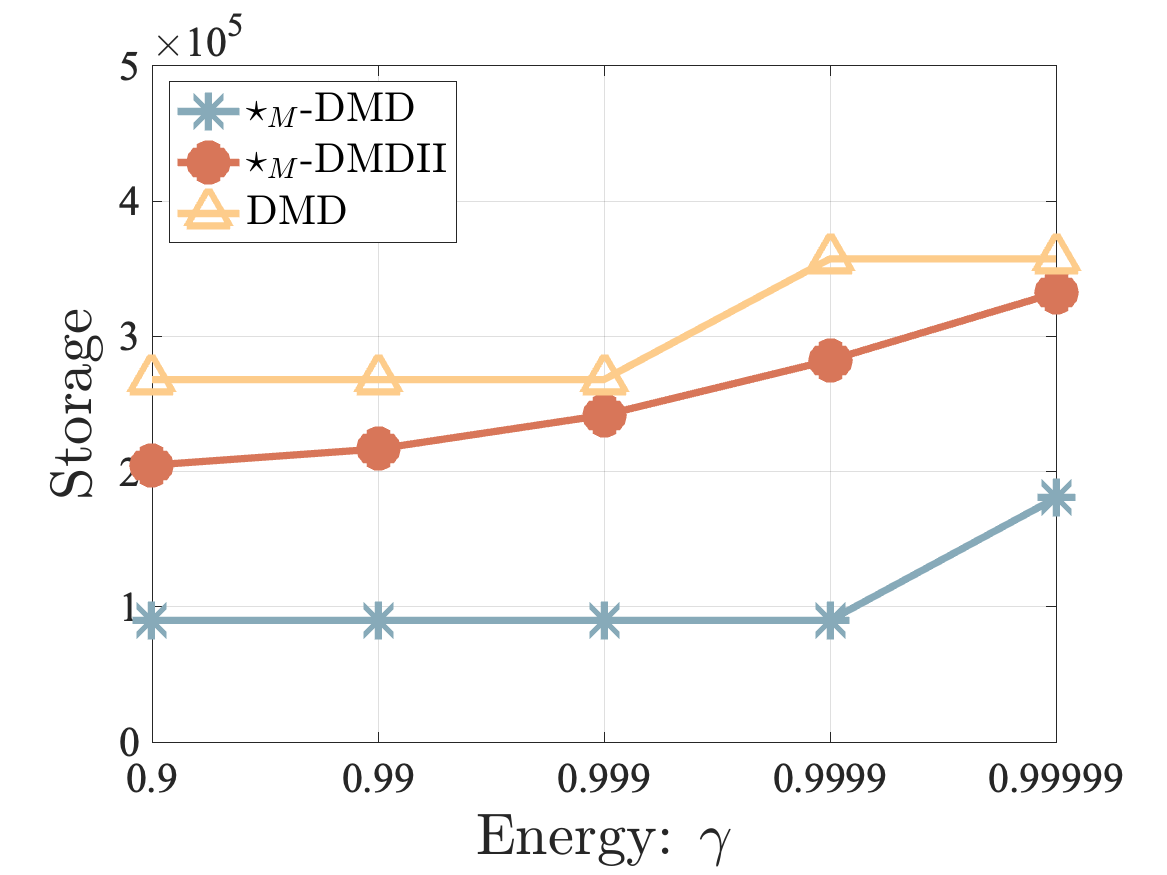}};
\node at (0, 5.2){\small $\B{M} = \text{DCT}$};
\node at (4.15, 5.2){\small $\B{M} = \text{DST}$};
\node at (8.3, 5.2){\small $\B{M} = \text{Data-driven}$};
\end{tikzpicture}
\caption{Performance comparison of the standard DMD, the $\starM$-DMD, and the $\starM$-DMDII methods on the cylinder data.}
\label{fig:cylinder}
\end{figure}

 Figure~\ref{fig:cylinder} compares the reconstruction results on the 2D cylinder flow dataset using the DMD, the $\starM$-DMD, and the $\starM$-DMDII methods. These reconstructions are performed at different energy values of $\gamma$ for the $\starM$-DMDII method from $0.9$ to $0.99999$. At these energy levels, we computed the required storage for the $\starM$-DMDII method first and then computed the corresponding ranks for the DMD and the $\starM$-DMD methods (labeled in yellow and blue, respectively) so that the three methods are compared at an equivalent level of storage. The experiments are repeated for the three different $\B{M}$ matrices used for the $\starM$-DMD and the $\starM$-DMDII methods. 

For all three choices of $\B{M}$, we can see that as the energy level or rank increases, the reconstruction RE decreases and the required storage increases for all three methods. In particular, the $\starM$-DMDII method results in a much lower reconstruction RE compared to the DMD and the $\starM$-DMD method while requiring a lower storage cost. 

\begin{figure}[!ht]
\centering
\begin{tikzpicture}
\node at (0,0) {\includegraphics[width=0.65\linewidth]{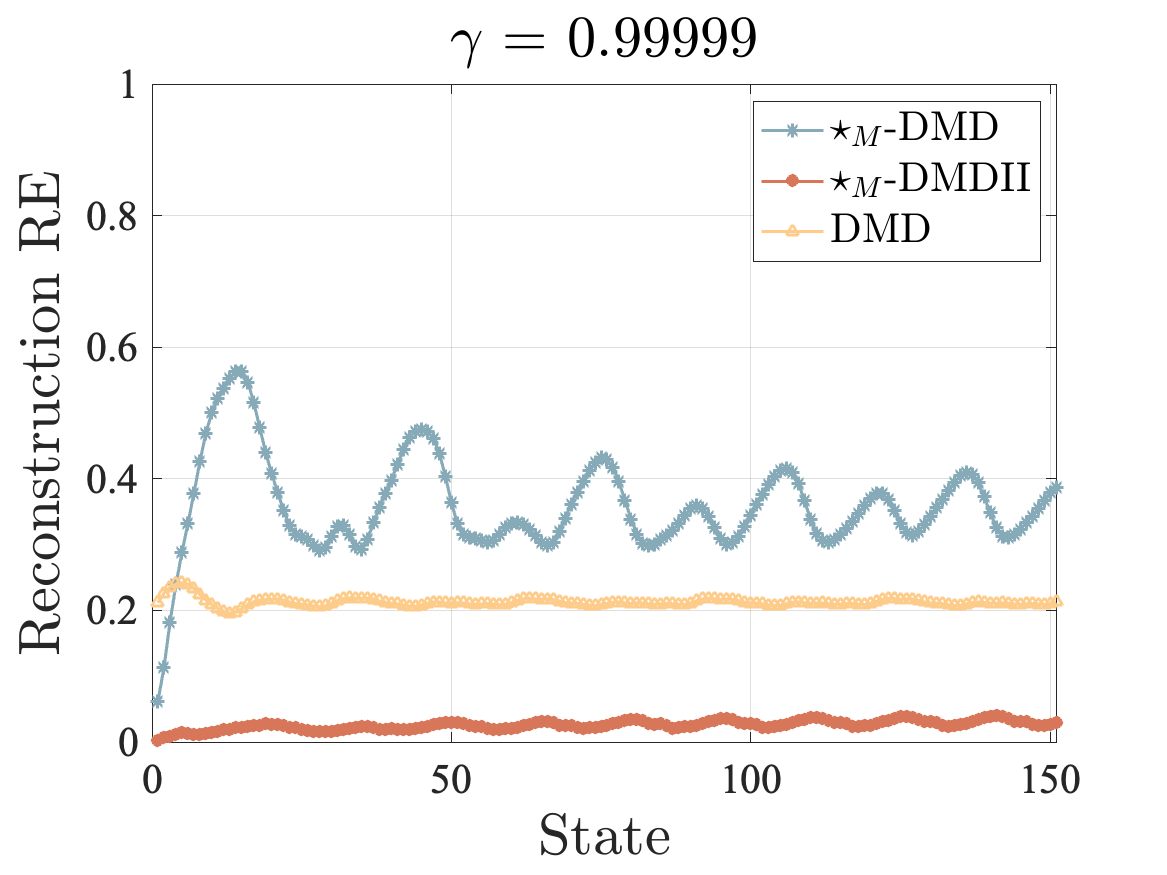}};
\end{tikzpicture}
\caption{State-wise reconstruction RE comparison of the three methods. The DMD $\text{rank}=4$, the $\starM$-DMD $\text{rank}=4$, and the $\starM$-DMDII energy $\gamma = 0.99999$. The orthogonal matrix $\B{M}$ is chosen to be the DCT matrix for the tensor-based methods.}
\label{fig:cyliner_state}
\end{figure}

 Figure~\ref{fig:cyliner_state} compares the reconstruction results of three methods at each state for a total of 150 snapshots. The DMD rank and the $\starM$-DMD rank are set to be $4$, and the $\starM$-DMDII energy is set to be $\gamma=0.99999$. Once again, the three methods have comparable storage costs. The transformation matrix $\B{M}$ is chosen to be the DCT matrix.

As we can see from the figure, the $\starM$-DMDII has much lower relative error RE overall. The  $\starM$-DMD reconstruction result has a smaller RE initially compared to the DMD reconstruction result, but is surpassed by the $\starM$-DMD method after a few states. Figure~\ref{fig:cyliner_snap} shows the flow vorticity plot of the first snapshots of the original data, the DMD reconstructed result, the $\starM$-DMD reconstructed result, and the $\starM$-DMDII reconstructed result. As can be seen from  the plots, although both the DMD and the $\starM$-DMD reconstructed results visually approximate the original flow data well, the $\starM$-DMDII reconstruction appears to recover the original flow data more accurately. 

\begin{figure}[!ht]
\centering
\begin{tikzpicture}
\node at (0,0) {\includegraphics[width=0.95\linewidth]{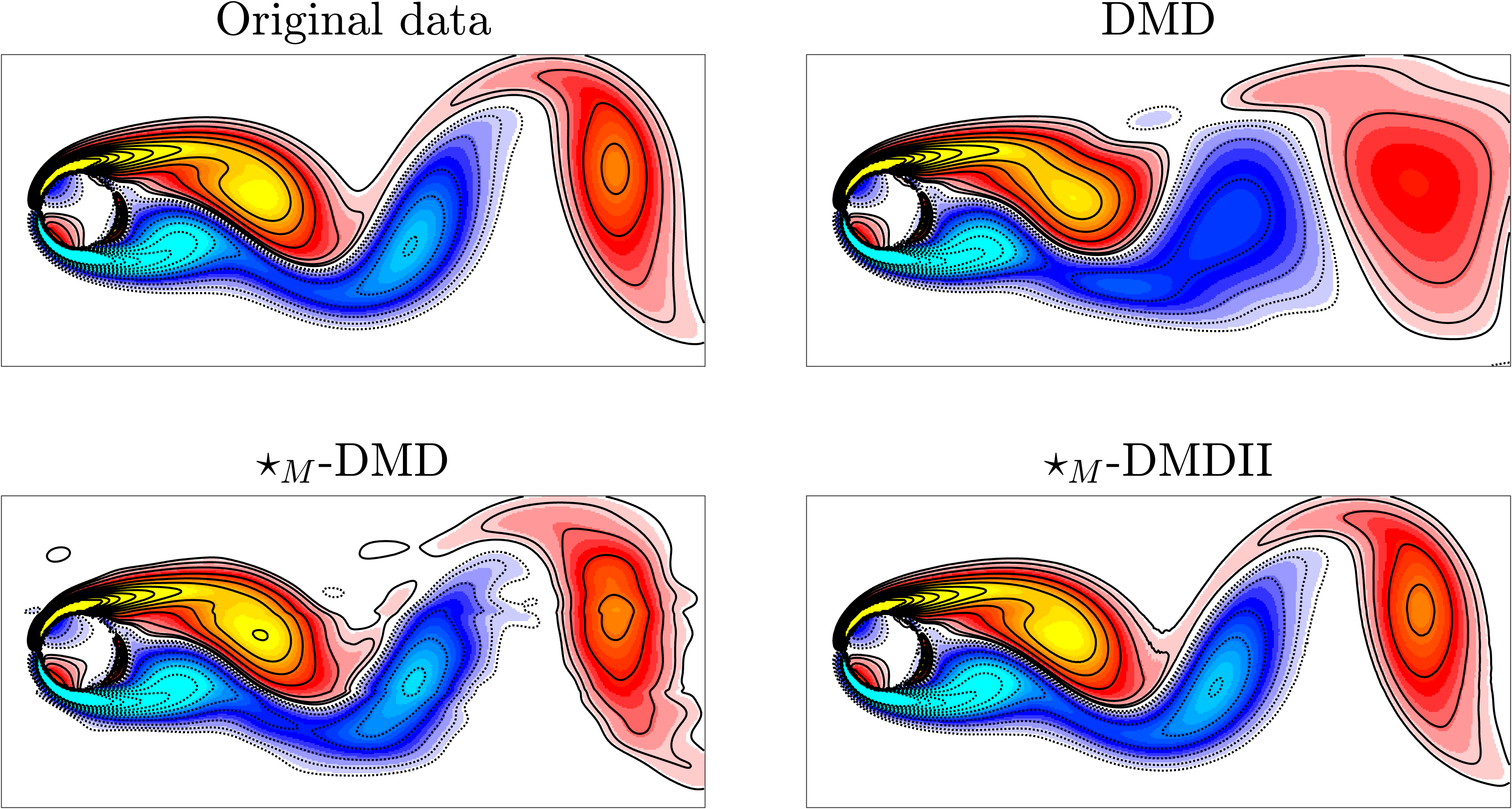}};
\end{tikzpicture}
\caption{Flow vorticity plot of the first snapshots of the original data, the DMD reconstructed result, the $\starM$-DMD reconstructed result, and the $\starM$-DMDII reconstructed result. The DMD $\text{rank}=4$, the $\starM$-DMD $\text{rank}=4$, and the $\starM$-DMDII energy $\gamma = 0.99999$. The orthogonal matrix $\B{M}$ is chosen to be the DCT matrix for the tensor-based methods.}
\label{fig:cyliner_snap}
\end{figure}

\subsection{Sea-surface temperature}
The NOAA Optimum Interpolation (OI) sea-surface temperature (SST) V2 dataset is publicly available \cite{reynolds2002improved}. The dataset consists of weekly temperature measurements from 1990 to 2016 on a grid of size $1^\circ \times 1^\circ$. In total, there are $1727$ snapshots each of size $360 \times 180$. We take the first $200$ snapshots and form a data tensor of size $360\times 200 \times 180$ where the snapshots form the lateral slices of the tensor. 

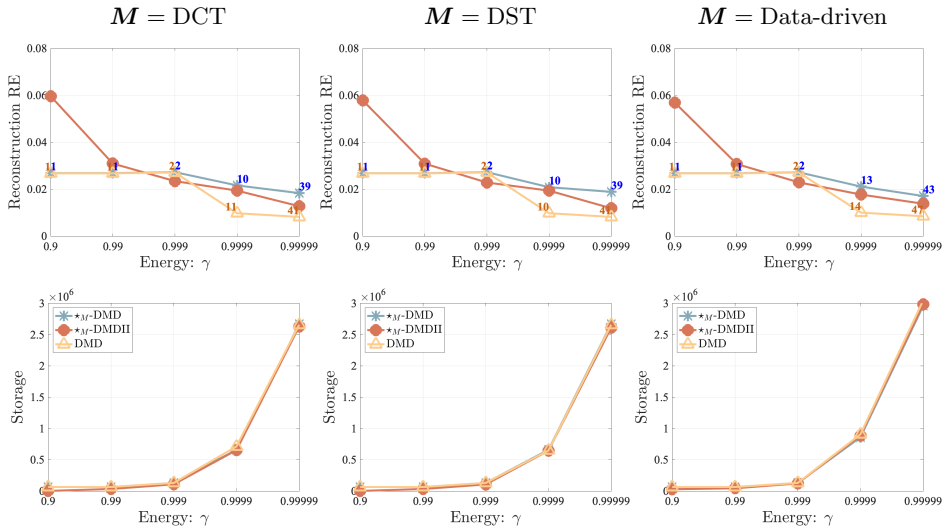
\begin{figure}[!ht]
\centering
\begin{tikzpicture}
\node at (0, 3.4) {\includegraphics[width=0.33\linewidth]{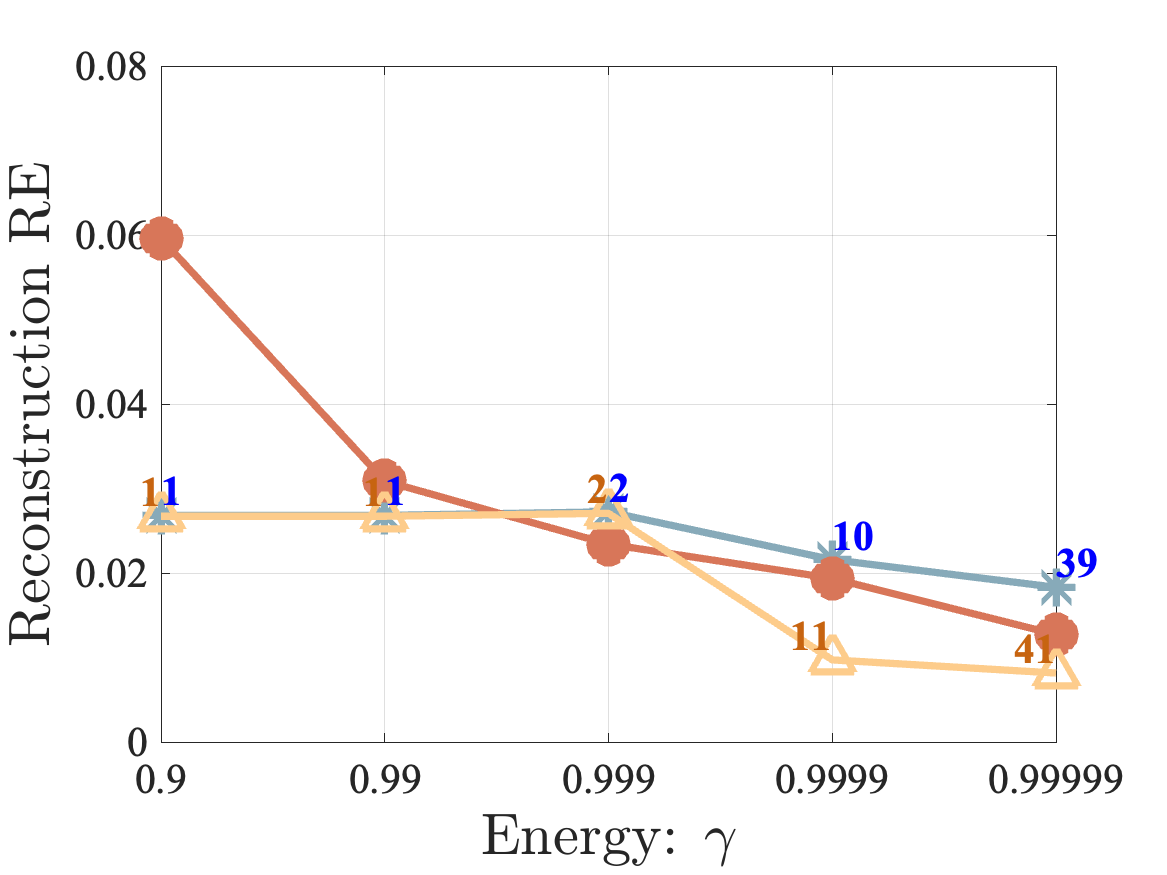}};
\node at (4.15, 3.4) {\includegraphics[width=0.33\linewidth]{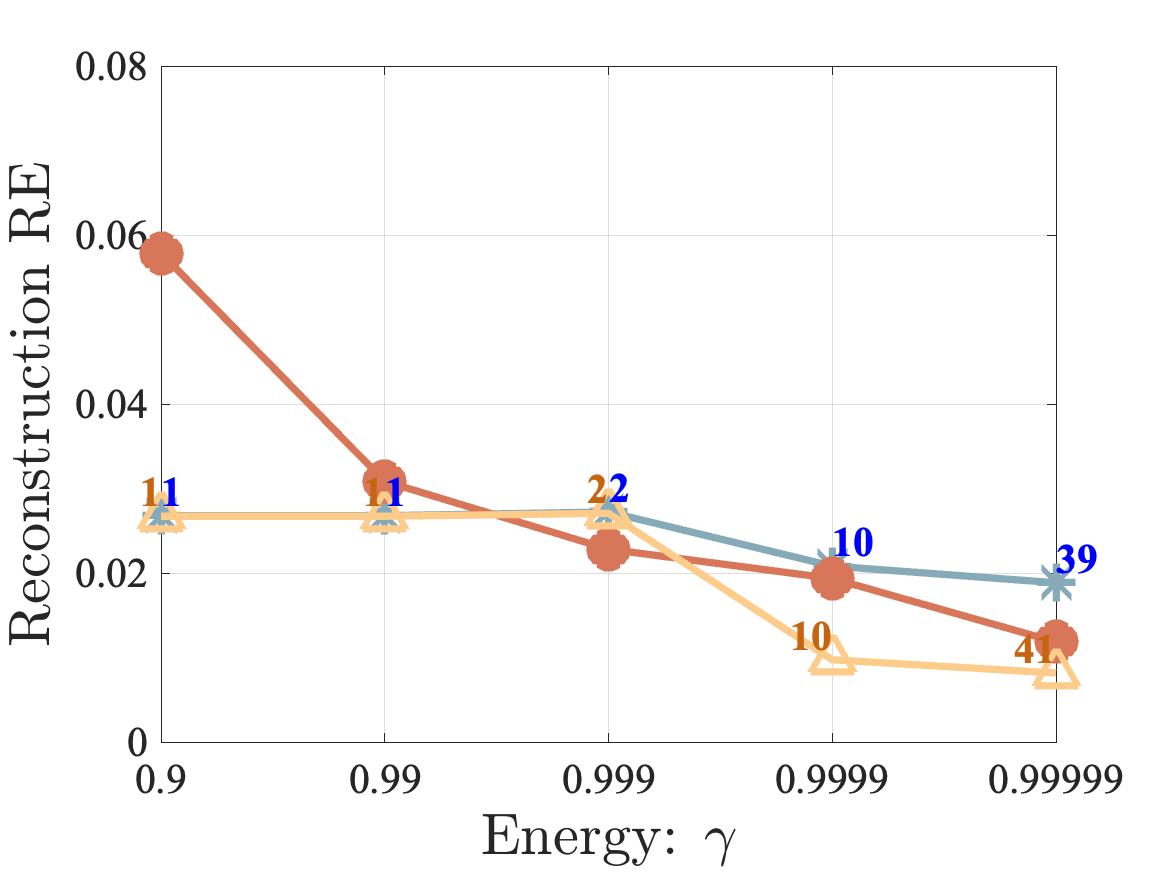}};
\node at (8.3, 3.4) {\includegraphics[width=0.33\linewidth]{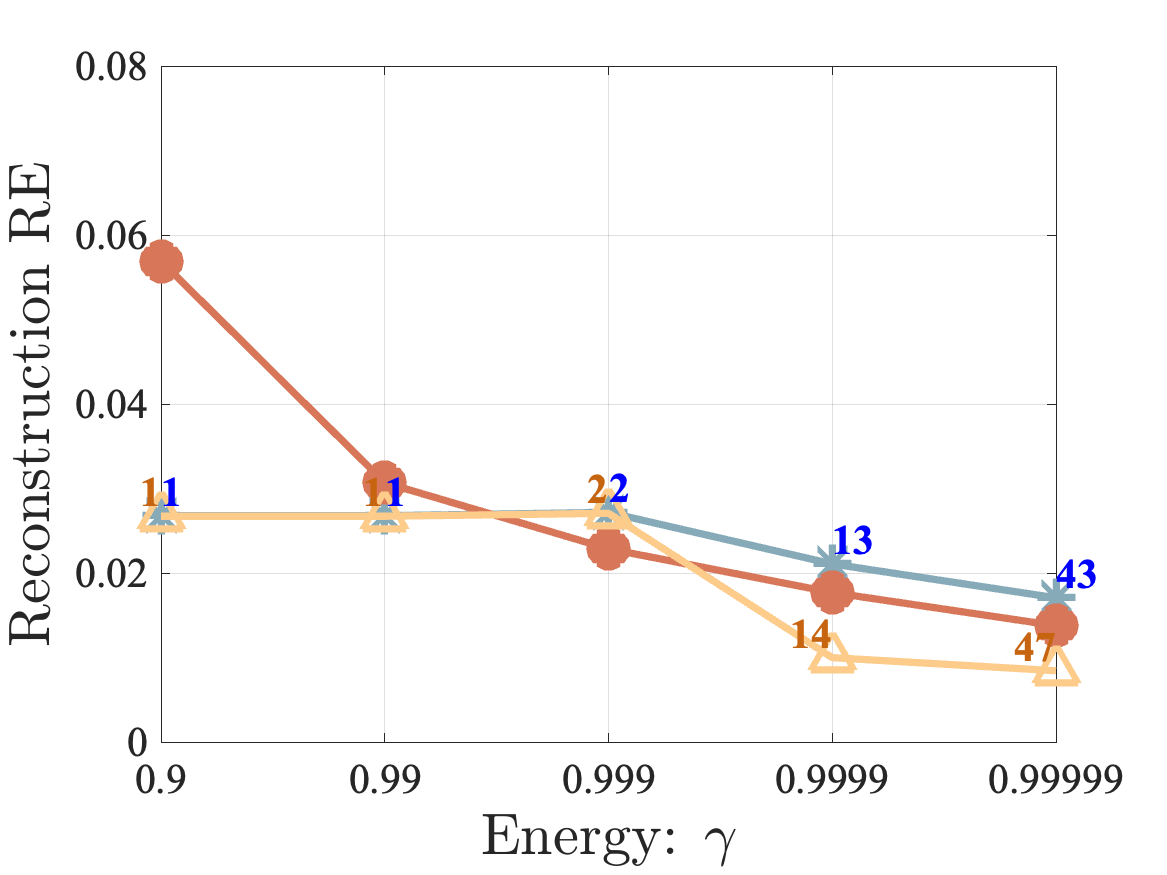}};

\node at (0,0) {\includegraphics[width=0.33\linewidth]{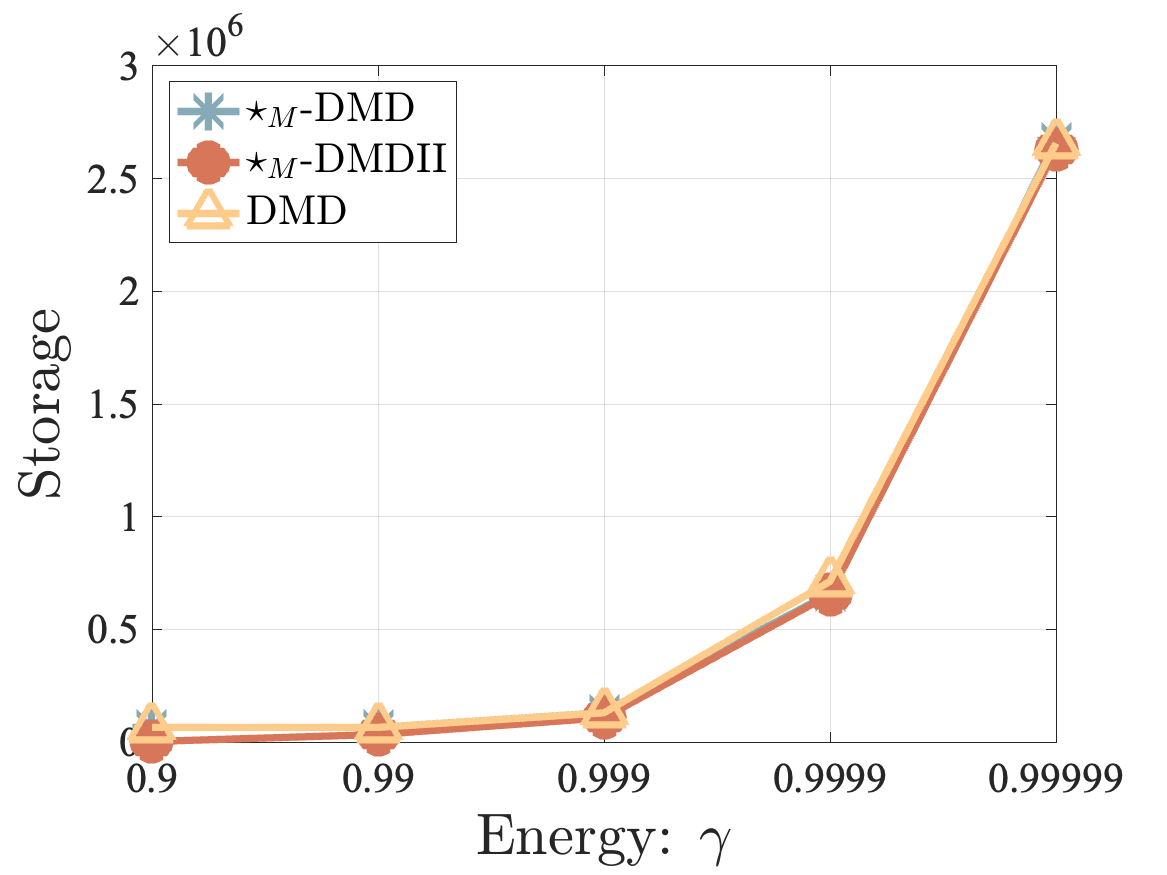}};
\node at (4.15,0) {\includegraphics[width=0.33\linewidth]{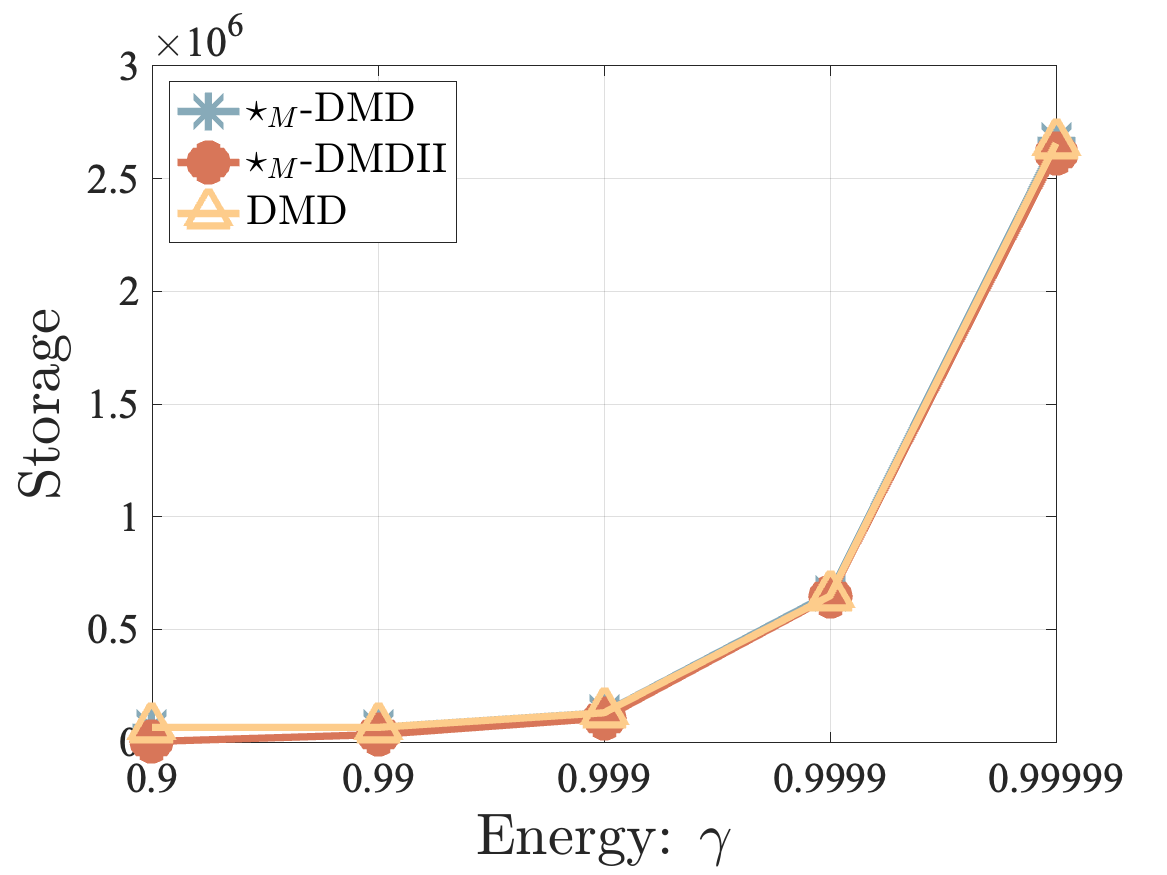}};
\node at (8.3,0) {\includegraphics[width=0.33\linewidth]{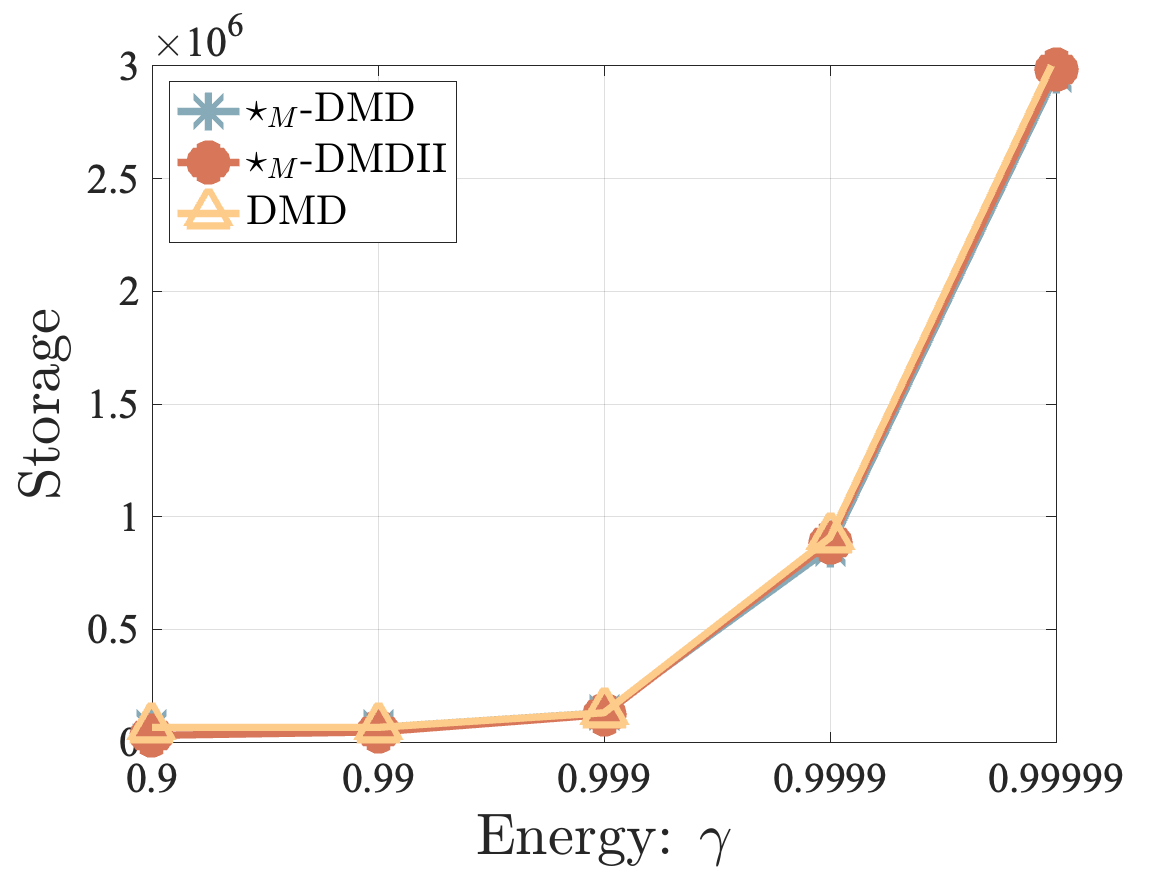}};
\node at (0, 5.2){\small $\B{M} = \text{DCT}$};
\node at (4.15, 5.2){\small $\B{M} = \text{DST}$};
\node at (8.3, 5.2){\small $\B{M} = \text{Data-driven}$};
\end{tikzpicture}
\caption{Performance comparison of the standard DMD, the $\starM$-DMD, and the $\starM$-DMDII methods on the SST dataset.}
\label{fig:sst_re_storage}
\end{figure}

Figure~\ref{fig:sst_re_storage} compares the reconstruction results of the three methods on the SST dataset. As before, we perform the reconstructions at different energy values of $\gamma$ for the $\starM$-DMDII method and determine the corresponding DMD rank and the $\starM$-DMD rank, so that all the methods have the same storage costs. We once again consider three different $\B{M}$ matrices: the DCT matrix, the normalized DST matrix, and the data-driven matrix. 

As we can see from this figure, the DMD reconstructed results have much smaller RE compared to the $\starM$-DMD method with $\text{rank}=1$ and the $\starM$-DMDII method with $\gamma=0.9$. For higher $\gamma$ values, the reconstruction results of the three methods are comparable, with the DMD method most accurate, followed by $\starM-$DMDII. We emphasize here that while the DMD method has a slight edge over $\starM-$DMDII, the latter is more computationally efficient, especially in a parallel environment.  

\begin{figure}[!ht]
\centering
\begin{tikzpicture}
\node at (0,0) {\includegraphics[width=0.65\linewidth]{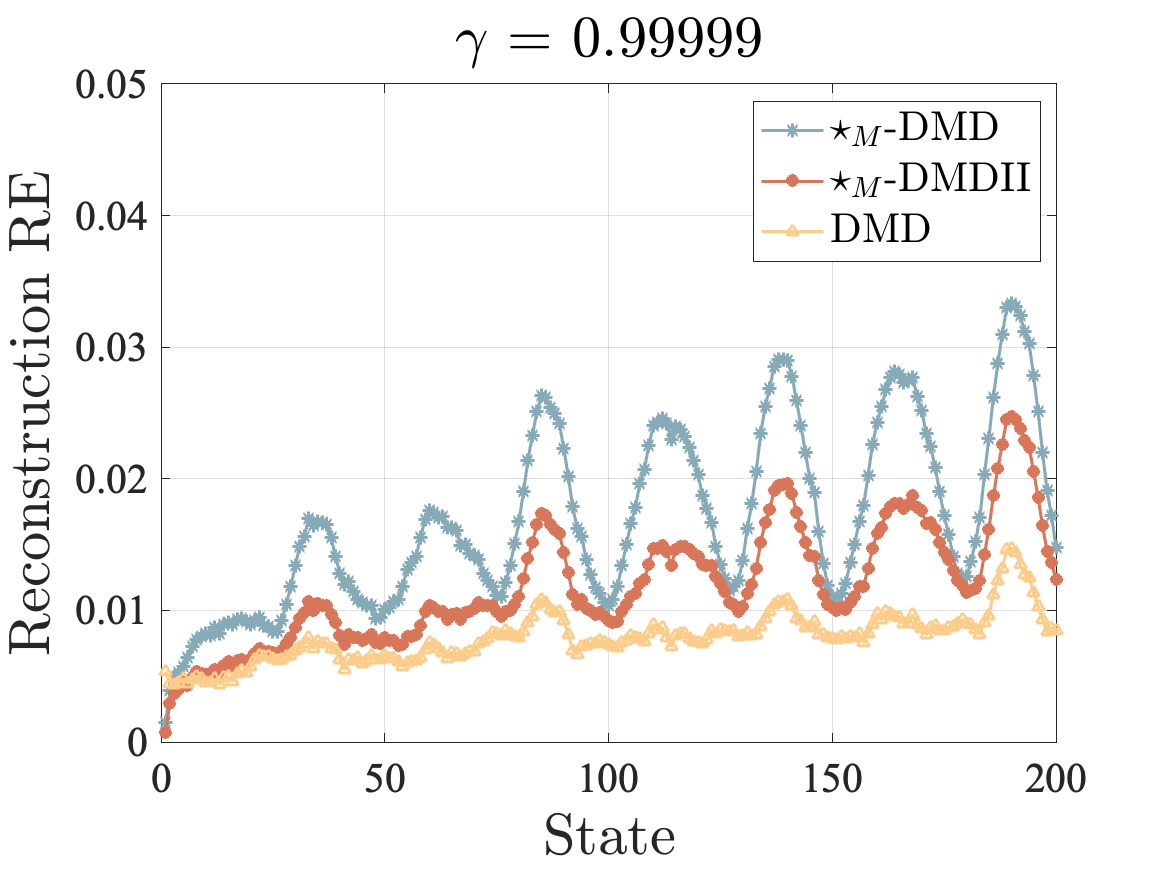}};
\end{tikzpicture}
\caption{State-wise reconstruction RE comparison of the three methods. The DMD $\text{rank}=41$, the $\starM$-DMD $\text{rank}=39$, and the $\starM$-DMDII energy $\gamma = 0.99999$. The orthogonal matrix $\B{M}$ is chosen to be the DCT matrix for the tensor-based methods.}
\label{fig:sst_state}
\end{figure}

\begin{figure}[ht]
\centering
\begin{tikzpicture}
\node at (0,0) {\includegraphics[width=0.95\linewidth]{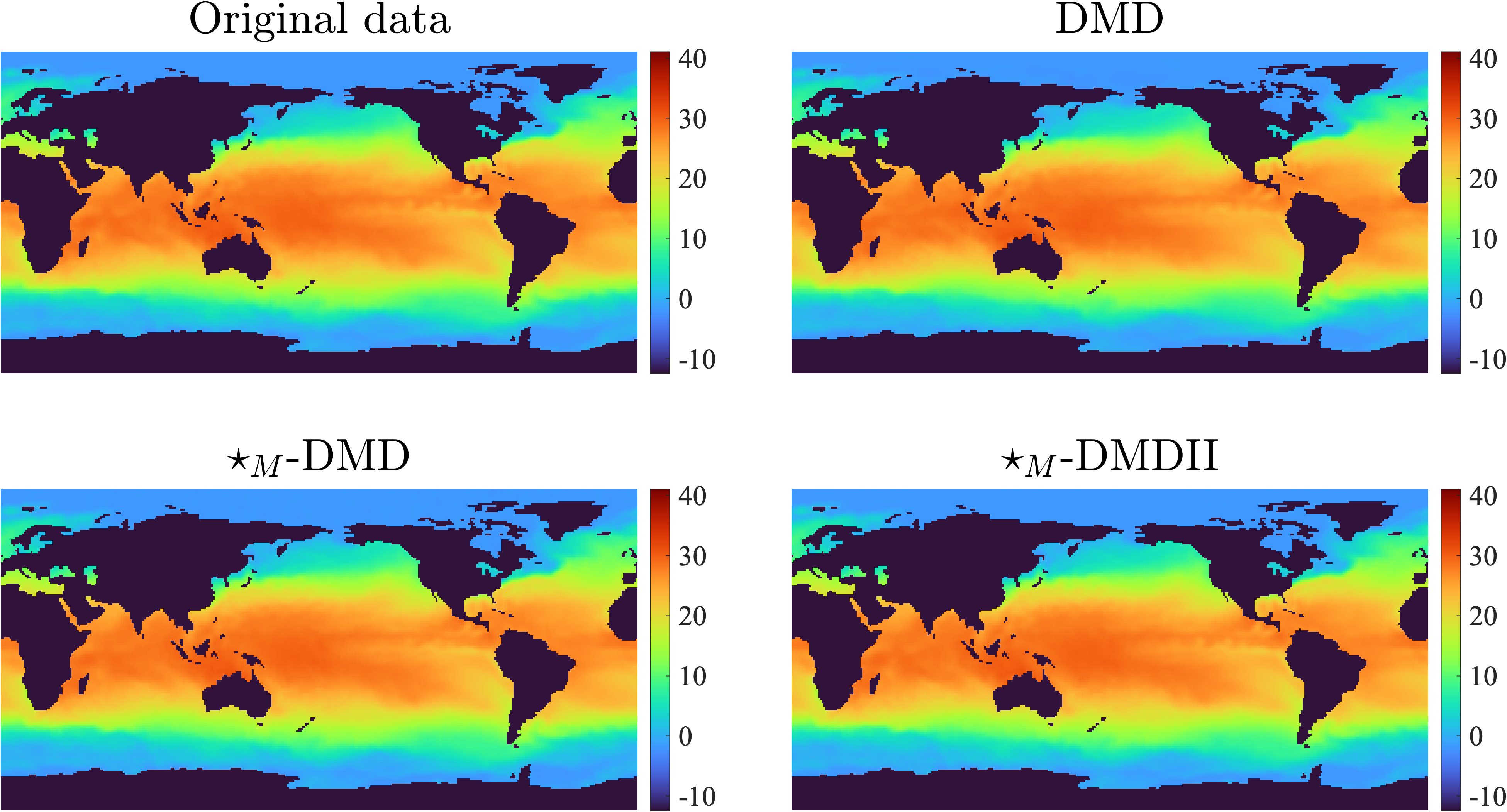}};
\end{tikzpicture}
\caption{The sea surface temperature maps of the first snapshot of the original data, the DMD reconstructed result, the $\starM$-DMD reconstructed result, and the $\starM$-DMDII reconstructed result. The DMD $\text{rank}=41$, the $\starM$-DMD $\text{rank}=39$, and the $\starM$-DMDII energy $\gamma = 0.99999$. The orthogonal matrix $\B{M}$ is chosen to be the DCT matrix for the tensor-based methods.}
\label{fig:sstmap}
\end{figure}

Figure~\ref{fig:sst_state} shows the state-wise reconstruction relative error of the SST dataset. The parameters are chosen as follows: the DMD $\text{rank}=41$, the $\starM$-DMD $\text{rank}=39$, and the $\starM$-DMDII energy $\gamma = 0.99999$. The orthogonal matrix $\B{M}$ is chosen to be the DCT matrix for the tensor-based methods. Overall, the DMD method performs better than the other two methods for this dataset, except that at the initial states, the $\starM$-DMDII method has the smallest reconstruction RE. Figure~\ref{fig:sstmap} shows the first snapshot of the original SST dataset and the reconstructed data by the three methods. As shown in the plots, the reconstructed results obtained by the three methods are comparable and well approximate the original data. 

\subsection{Streaming on Cylinder Flow} The second cylinder flow dataset can be accessed publicly\footnote{Cylinder Flow with von Karman Vortex Sheet: \url{https://cgl.ethz.ch/research/visualization/data.php}}. This dataset simulates a two-dimensional viscous flow past a cylinder. The fluid is introduced from the left side of a channel, which is bounded by solid walls subject to slip boundary conditions. The first 200 snapshots were used for this experiment. Each snapshot is of size $640\times 80$. The dataset was divided into $b = 20$ batches, with each batch consisting of $10$ consecutive snapshots. For comparison, we consider a streaming version of DMD. 

We used the same parameter $\rho_{\max}$ for both the streaming matrix DMD and $\starM$-DMDII. Both the streaming algorithms have nearly identical storage costs, in terms of the sketches and the random matrices. The DCT matrix is used in the $\starM$-computations.

Figure~\ref{fig:stream_DMD} (a) compares the batch reconstruction results of the streaming DMD method and the streaming $\starM$-DMDII method. That is, we process each batch sequentially as in the loop determined by Lines 4-11 in Algorithm~\ref{alg:starMsteaming} and determine the reconstruction $\widetilde{\T{C}}_j$ for each batch $1 \le j\le 20$. We then compute the relative reconstruction error $\|\T{C}_j - \widetilde{\T{C}}_j\|_F/ \|\T{C}_j\|_F$.  Figure~\ref{fig:stream_DMD} (b) shows the state-wise reconstruction RE of the two methods, after all the batches have been processed. As we can see from the figures, the streaming $\starM$-DMDII outperforms the streaming DMD method for this dataset in terms of relative error. 

\begin{figure}[ht]
\centering
\begin{tikzpicture}
\node at (0,0) {\includegraphics[width=0.45\linewidth]{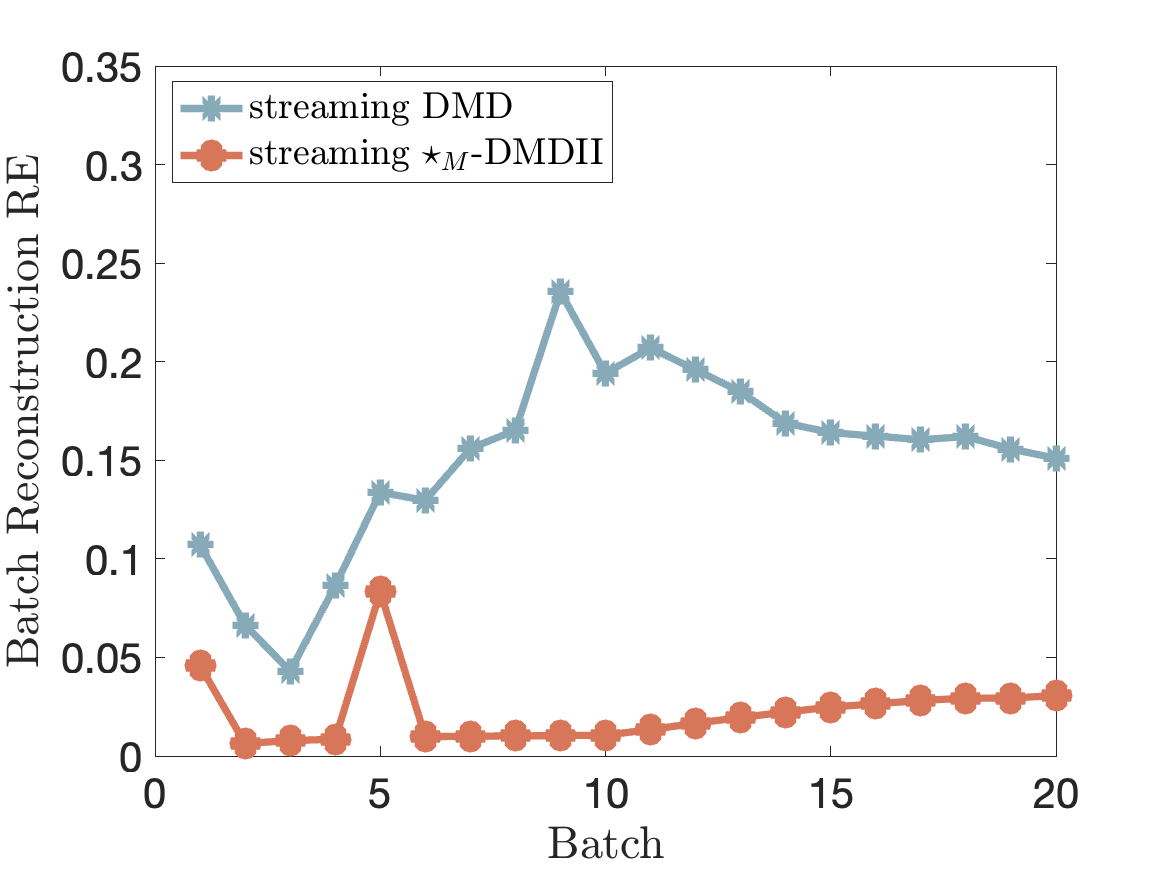}};
\node at (6,0) {\includegraphics[width=0.45\linewidth]{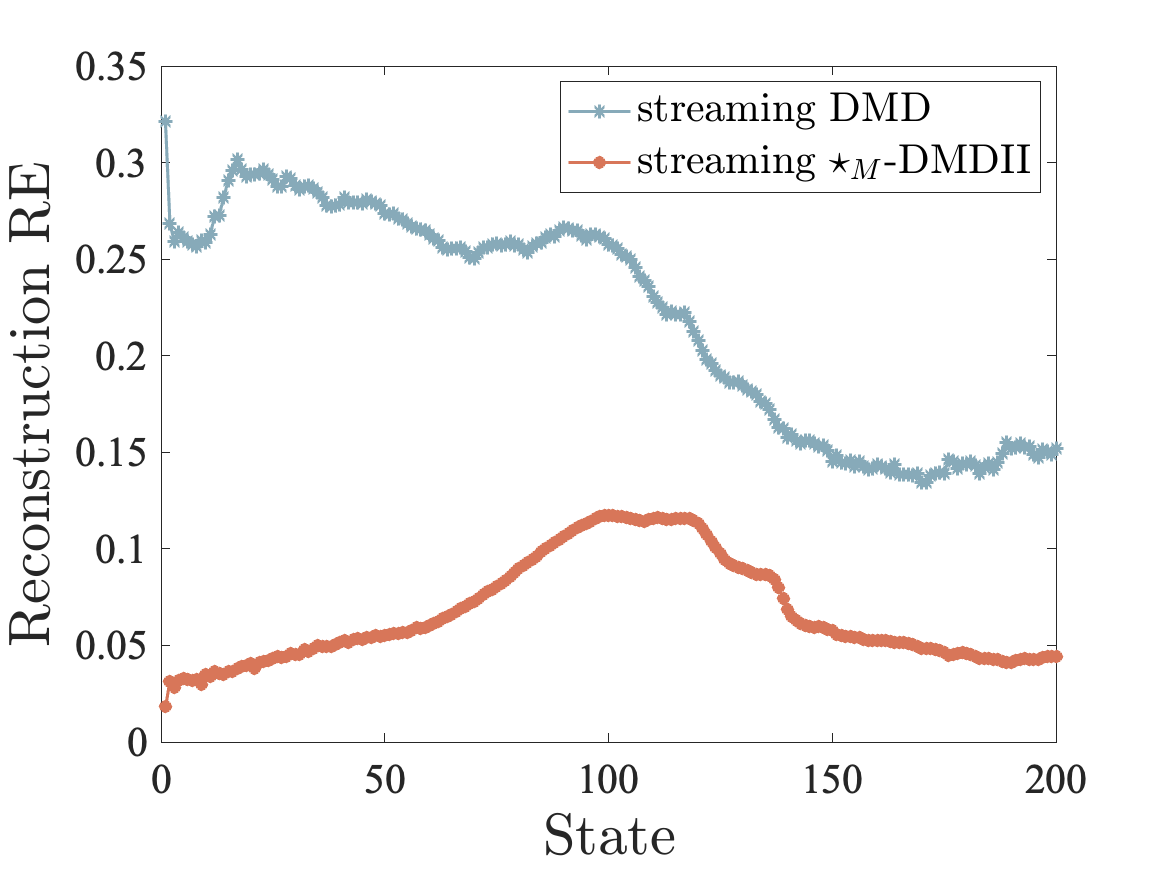}};
\node at (-2.7,2.2) {\small (a)};
\node at (3.1,2.2) {\small (b)};
\end{tikzpicture}
\caption{DMD $\rho_{\max}=20$, $\starM$-DMDII energy $\gamma = 0.99999$. $\B{M}=\text{DCT}$.}
\label{fig:stream_DMD}
\end{figure}

Figure~\ref{fig:stream_snapshots} shows the flow vorticity plots of the first (left column) and the $50$th (right column) snapshots of the original data, the streaming DMD reconstructed result, and the streaming $\starM$-DMDII reconstructed result. As we can see from the plots, the streaming $\starM$-DMDII method shows more accurate reconstructions of the original dataset compared to the streaming DMD method. 

\begin{figure}[ht]
\centering
\begin{tikzpicture}
\node at (-3.2, 2.8) {Snapshot $1$};
\node at (3.2, 2.8) {Snapshot $50$};
\node at (0,0) {\includegraphics[width=0.95\linewidth]{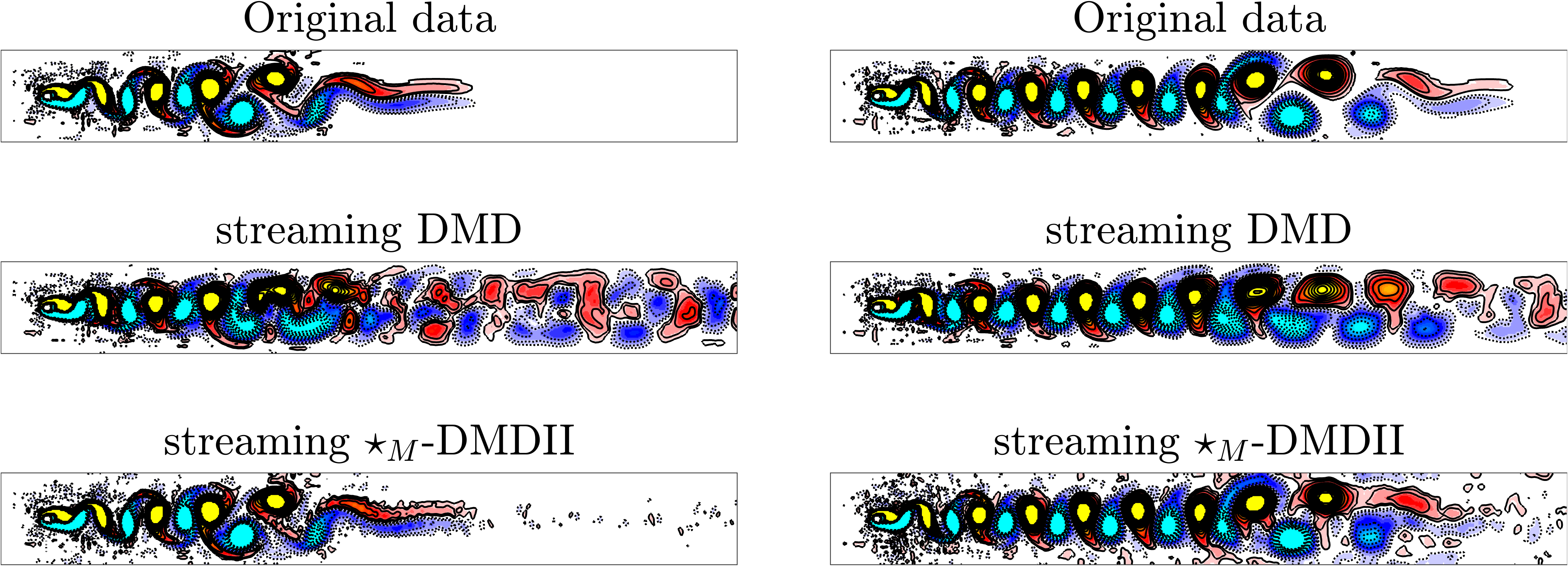}};
\end{tikzpicture}
\caption{Data snapshots: DMD $\rho_{\max}=9$, $\starM$-DMDII energy $\gamma = 0.99999$. $\B{M}=\text{DCT}$.}
\label{fig:stream_snapshots}
\end{figure}

\section{Conclusion and discussion}\label{sec:conc}
This paper presents a new tensor-based approach for DMD based on the $\starM$-product. The method exploits the multidimensional nature of the snapshots. A computational framework for the $\starM$-DMD is given which is attractive from a computational point of view since it can exploit inherent parallelism in the $\starM$-framework. The approach can be interpreted from the regression viewpoint of DMD, as a form of regression constrained over a matrix subspace defined by the $\starM$-product. Thus, this is similar to the physics-informed DMD approach as well. We also consider the streaming setting, where the snapshots arrive sequentially in batches and give efficient randomized algorithms for updating the $\starM$-DMD in this setting. Numerical experiments on several flow-based datasets show that the $\starM$-DMD have comparable or greater accuracy compared to the standard DMD approach, when the storage cost of the methods is held equal. 

There are several avenues for future work. While the experiments focused on flow problems in two spatial dimensions, the framework extends readily to three spatial dimensions. One straightforward approach for three-spatial dimensions with $N=n_xn_yn_z$ grid points, is to take in the notation of Section~\ref{ssec:tenreg} $m= n_x$, $n = n_yn_z$, and $p = T$. For $d= 3$ spatial dimensions, we also take $\B{M}$ as a Kronecker product of the form $\B{M} = \B{M}_Y \otimes \B{M}_Z$. See~\cite{kilmer2021tensor} for details on extensions. We intend to explore the computational issues associated with $\starM$-DMD involving three spatial dimensions in future work. There are several variants of the DMD that can be viewed from the Galerkin perspective (e.g., extended DMD)~\cite[Section 2.2.2]{colbrook2023multiverse}. It would be interesting to explore the connections between these approaches and the $\starM$-DMD approach. Finally, it may be worth exploring optimal choices for the matrix $\B{M}$ that defines the $\starM$-product. A possible way forward is the approach in~\cite{newman2025optimal}.

\section{Acknowledgements}
AKS would like to thank Mohammad Farazmand for helpful conversations on the formulation and the numerical results.

\bibliography{refs}
\bibliographystyle{abbrv}
\end{document}